\documentclass[10pt,a4paper]{amsart}
\usepackage[english]{babel}
\usepackage{amsmath,amsthm,mathrsfs}
\usepackage{amssymb}
\usepackage{aliascnt}
\usepackage{comment}
\usepackage{todonotes}
\usepackage{xspace}
\usepackage{nicefrac}
\usepackage{diagrams}
\usepackage{marvosym}

\DeclareMathOperator{\dom}{dom}

\DeclareMathOperator{\cof}{cof}

\DeclareMathOperator{\crit}{crit}

\DeclareMathOperator{\rank}{rank}

\DeclareMathOperator{\cf}{cf}

\DeclareMathOperator{\sk}{Sk}

\DeclareMathOperator{\Ult}{Ult}

\newcommand{\chang}{\twoheadrightarrow}

\newtheorem{theorem}{Theorem}
\newaliascnt{example}{theorem}

\aliascntresetthe{example}
\newaliascnt{fact}{theorem}

\aliascntresetthe{fact}
\newaliascnt{corollary}{theorem}
\newtheorem{corollary}[corollary]{Corollary}
\aliascntresetthe{corollary}
\newaliascnt{lemma}{theorem}
\newtheorem{lemma}[lemma]{Lemma}
\aliascntresetthe{lemma}
\newaliascnt{claim}{theorem}
\newtheorem{claim}[lemma]{Claim}
\aliascntresetthe{claim}
\newtheorem{prop}[theorem]{Proposition}
\theoremstyle{definition}
\newaliascnt{definition}{theorem}

\aliascntresetthe{definition}

\newtheorem*{theorem*}{Theorem}

\newtheorem*{example*}{Example}

\newtheorem*{definition*}{Definition}
\newtheorem*{question*}{Question}

\DeclareMathOperator{\ran}{ran}
\DeclareMathOperator{\ot}{ot}

\DeclareMathOperator{\add}{Add}
\DeclareMathOperator{\col}{Col}
\DeclareMathOperator{\ord}{Ord}

\DeclareMathOperator{\id}{id}
\DeclareMathOperator{\ns}{NS}

\DeclareMathOperator{\wcc}{wCC}
\DeclareMathOperator{\cc}{CC}
\DeclareMathOperator{\ccc}{clubCC}

\newcommand{\calA}{\mathcal{A}}
\newcommand{\calB}{\mathcal{B}}
\newcommand{\calC}{\mathcal{C}}
\newcommand{\calD}{\mathcal{D}}
\newcommand{\calE}{\ensuremath{\mathcal{E}}}

\newcommand{\calG}{\ensuremath{\mathcal{G}}}

\newcommand{\calI}{\ensuremath{\mathcal{I}}}
\newcommand{\calJ}{\ensuremath{\mathcal{J}}}

\newcommand{\calS}{\mathcal{S}}

\newcommand{\calT}{\mathcal{T}}
\newcommand{\calU}{\mathcal{U}}
\newcommand{\calW}{\mathcal{W}}

\newcommand{\rest}{\restriction}
\newcommand{\p}{\mathcal{P}}
\newcommand{\la}{\langle}
\newcommand{\ra}{\rangle}
\title{Weak saturation properties and side conditions}
\author{Monroe Eskew}
\thanks{This work was supported by the Austrian Science Fund (FWF) through Project P34603.}
\begin{document}
\maketitle

\begin{abstract}
Towards combining ``compactness'' and ``hugeness'' properties at $\omega_2$, we investigate the relevance of side-conditions forcing.  We reduce the upper bound on the consistency strength of the weak Chang's Conjecture at $\omega_2$ using Neeman's forcing.  But we find a barrier to the applicability of these methods to our problem and give a counterexample to a claim of Neeman about the effects of iterating such forcing.
\end{abstract}

\section{Introduction}

Can ``compactness'' and ``hugeness'' properties coexist at small cardinals?  More specifically, can $\omega_2$ satisfy the tree property and also carry a saturated ideal?     This question is natural for anyone concerned with ``large cardinal properties'' that can hold of small cardinals.  It turns out to be surprisingly difficult to answer.  Cox and the author showed in \cite{coxeskew} that if this situation is consistent, then (a) it requires the continuum to be at least $\omega_3$, and (b) it cannot be forced via a ``Kunen-style'' generic lifting of an almost-huge embedding, which is presently the only known method to produce models with saturated ideals on $\omega_2$.  This means that, if these properties of $\omega_2$ are mutually consistent, a proof will require rather novel methods.
% for forcing saturated ideals on successor cardinals.

A promising idea is to look towards the newer techniques of forcing with sequences of models.  The basic idea of using models as side-conditions to control a forcing construction was invented by Todor\v{c}evi\'{c} \cite{todo89}.  Work of Friedman \cite{friedman} and Mitchell \cite{mitchellapp} expanded upon this idea, and influential work of Neeman \cite{neeman} made further advances, providing an elegant framework using two types of models.  Mitchell \cite{mitchellapp} and Neeman \cite{neeman} also suggested iterating such forcing to obtain the tree property on successive cardinals.  The author realized that, when combined with very large cardinals, this could also simultaneously force an approximation of saturation at $\omega_2$, namely the weak Chang's Conjecture $\wcc(\omega_2,\cof(\omega_1))$.

However, we ultimately found some combinatorial constraints that make the Neeman technique unsuitable for combining compactness and hugeness properties at successor cardinals.  Thus the present work can be seen as adding to the limitative results of \cite{coxeskew}.  In particular, Theorem \ref{wccsquare} says that if the continuum is at most $\omega_2$ and $\wcc(\omega_2,\cof(\omega_1))$ holds, then there is a special $\omega_2$-Aronszajn tree.  This implies that an iteration of side-conditions forcing cannot work as claimed in \cite{neeman} in full generality.

Another method to investigate would be that developed by Mohammadpour and Veli\v{c}kovi\'c \cite{mv}, which obtains the tree property at $\omega_2$ and $\omega_3$ while strengthening the celebrated consistency result of Mitchell \cite{mitchellapp}.  Since their model satisfies $2^\omega = \omega_3$, it evades the constraint imposed by Theorem \ref{wccsquare}.  However, Mohammadpour showed that their poset always forces $\neg\wcc(\omega_2,\cof(\omega_1))$, and thus that there is no saturated ideal on $\omega_2$.  We reproduce his argument here with his kind permission.

Our results here are not entirely negative.  We uncover some new combinatorial facts that add to the tension between ``compactness'' and ``hugeness'' properties.  We isolate a new large cardinal notion, ``ambitious'' cardinals, and show several equivalent characterizations.  We show how Neeman forcing may introduce the weak Chang's Conjecture and derive a forcing characterization of ambitiousness.  We reduce the known upper bound on the consistency strength of $\wcc(\omega_2,\cof(\omega_1))$.

The structure of the paper is as follows.  In \S\ref{sec:weaksat}, we discuss saturation properties of ideals, versions of Chang's Conjecture, and equivalences in terms of  properties of collections of elementary submodels.
In \S\ref{sec:ccsquare}, we show some new combinatorial results connecting versions of Chang's Conjecture and square principles.
  In \S\ref{sec:neeman}, we introduce a modified version of Neeman forcing and lay out its important properties.  In \S\ref{sec:ee}, we discuss a large cardinal above $0^\sharp$ but below measurable that can be used with Neeman forcing to obtain weak presaturation at $\omega_1$.  In \S\ref{sec:wccomega_2}, we obtain a new upper bound on $\wcc(\omega_2,\cof(\omega_1))$.  In \S\ref{sec:ambitious}, we introduce ambitious cardinals and discuss their connection with so-called Magidor models.  
Finally, in \S\ref{sec:hopeless}, we give a forcing characterization of ambitiousness, discuss the implications for iterated Neeman forcing, and sketch Mohammadpour's argument that the forcing of \cite{mv} does not allow hugeness properties of $\omega_2$.
  %Finally, in \S\ref{sec:wcctp}, we show how to combine wCC at $\omega_2$ with the tree property and derive some limitative results.

We assume familiarity with the basics of stationary sets, large cardinals, forcing, and the notions of properness and strong properness for a class of models.  We also assume a strong familiarity with \cite{neeman}, as we will rely on it.  

To fix some notation, if $\kappa$ is a cardinal and $X$ is a set, we use $[X]^{<\kappa}$ to denote $\{ z \subseteq X : |z| < \kappa \}$, and if $\kappa \subseteq X$, we use $\p_\kappa(X)$ to denote $\{ z \subseteq X : |z| < \kappa \wedge z \cap \kappa \in \kappa \}$.  For a set of ordinals $x$, we let $\ot(x)$ denote its order-type.
  $H_\kappa$ denotes the collection of all sets of hereditary cardinality less than $\kappa$.  For a structure $\frak A$ on $X$ carrying a well-order and $Y \subseteq X$, we write $\sk^{\frak A}(Y)$ for the Skolem hull of $Y$, i.e.\ the closure of $Y$ under the definable Skolem functions.

\section{Weak saturation properties}
\label{sec:weaksat}
Foreman \cite{foreman} proved an equivalence between a normal fine ideal being saturated and a statement about elementary submodels.  Inspired by this, we identify some weakenings of saturation that also have equivalences in terms of elementary submodels, in the hopes that a forcing such as Neeman's, which is built from sequences of models, may be able to achieve desired results by directly manipulating properties of collections of elementary submodels.  Let us first recall some basic notions.

$\calI$ is an \emph{ideal} over a set $Z$ when $\calI \subseteq \p(Z)$, and $\calI$ is closed under subsets and pairwise unions.  For a cardinal $\kappa$, $\calI$ is \emph{$\kappa$-complete} when it is closed under unions of size ${<}\kappa$.  We say a set $A \subseteq Z$ is \emph{$\calI$-positive} when it is not in $\calI$.  We denote the collection of all such sets as $\calI^+$, and we let $\calI^d$ denote the filter dual to $\calI$.  If $Z \subseteq \p(X)$, we say that $\calI$ is \emph{fine} when $\{ z \in Z : x \in z \} \in \calI^d$ for all $x \in X$, and \emph{normal} when for every $\calI$-positive $A$ and every function $f : A \to X$ such that $f(z) \in z$ for all $z \in A$, $f$ is constant on an $\calI$-positive set. 
%there is $x \in X$ such that $f^{-1}[\{x\}]$ is $\calI$-positive.  
A normal fine ideal $\calI$ on $Z \subseteq \p(X)$ is said to be \emph{saturated} when for every collection $\calA$ of $\calI$-positive sets that have pairwise intersection in $\calI$, $|\calA| \leq |X|.$  In other words, the Boolean algebra $\p(Z)/\calI$ has the $|X|^+$-chain condition.

A set $Z \subseteq \p(X)$ is \emph{stationary} when for all $F : X^{<\omega} \to X$, there is $z \in Z$ closed under $F$.  The set of all $z \in \p(X)$ 
closed under such an $F$ is called a \emph{closed-unbounded set} or a \emph{club}.  It is well-known that for all stationary sets $Z$, the collection of all nonstationary subsets of $Z$ forms the smallest normal fine ideal over $Z$.  The ideal of nonstationary subsets of $Z$ will be denoted by $\ns_Z$, or just $\ns$ when $Z$ is clear from context.  If $Z$ is a regular cardinal $\kappa$, this notion of club coincides with the usual order-topological notion.

Suppose $\pi : Z_1 \to Z_0$ and $\calJ$ is an ideal on $Z_1$.  The map $\pi$ induces an ideal $\calI$ on $Z_0$, where we put $A \in \calI$ iff $\pi^{-1}[A] \in \calJ$.  If $\calJ$ is $\kappa$-complete, then so is $\calI$.  If $Z_i \subseteq \p(X_i)$ for $i<2$, $X_0 \subseteq X_1$, and $\pi$ is the map $z \mapsto z \cap X_0$, then we say that $\calI$ is the \emph{canonical projection} of $\calJ$.  In this case, if $\calJ$ is normal and fine, then so is $\calI$.
For a normal fine ideal $\calI$ on $Z \subseteq \p(X)$ and $\theta>2^{|Z|}$, Foreman \cite{foreman} defined a model $M \prec H_\theta$ to be \emph{$\calI$-good} when $\calI \in M$ and $M \cap X \in A$ for every $A \in \calI^d \cap M$.  He showed:

\begin{theorem}[Foreman]
Suppose $\calI$ is a normal fine ideal on $Z \subseteq \p(X)$, and $\theta>2^{|Z|}$ is regular.
\begin{enumerate}
\item If $\calG$ is the set of $\calI$-good $M \prec H_\theta$, then $\calG$ is stationary, and $\calI$ is the canonical projection of $\ns \rest \calG$.
\item If $|X| = |Z|$, then $\calI$ is saturated if and only if for every $\calI$-good $M \prec H_\theta$ and every maximal antichain $\calA \subseteq \p(Z)/\calI$ in $M$, there is $a \in \calA \cap M$ such that $M \cap X \in a$.
\end{enumerate} 
\end{theorem}

It is well-known that saturated ideals produce well-founded generic ultrapowers, which are moreover highly closed from the perspective of the generic extension, and that  these properties can be obtained from the following weaker notion.
A normal ideal $\calI$ on a regular cardinal $\kappa$ is called \emph{presaturated} when for 
%all $\mu<\kappa$, 
all $\calI$-positive $B$, and all collections of antichains $\la \calA_i : i < \kappa \ra$, there is an $\calI$-positive $C \subseteq B$ such that $|\{ a \in \calA_i : a \cap C \in \calI^+ \}| \leq \kappa$ for all $i<\kappa$.
It is not hard to show that forcing with a presaturated ideal on $\kappa$ preserves $\kappa^+$.  Thus if $\kappa$ is a successor cardinal, any  generic embedding arising from such an ideal must send $\kappa$ to $(\kappa^+)^V$.  

Woodin \cite{woodin} calls a normal ideal $\calI$ on a successor cardinal $\kappa$ \emph{weakly presaturated} just when the above property occurs, i.e.\ $\Vdash_{\p(\kappa)/\calI} j_{\dot G}(\kappa) = (\kappa^+)^V$.  This can be explicated in terms of \emph{canonical functions}.
Recall that the canonical functions on a regular cardinal $\kappa$ represent ordinals $\alpha<\kappa^+$ in any generic ultrapower coming from a normal ideal on $\kappa$, and they can be defined (up to equivalence modulo clubs) as follows:  for $\alpha<\kappa^+$, pick any surjection $\sigma_\alpha : \kappa \to \alpha$, and define the $\alpha^{th}$ canonical function $\psi_\alpha$ by $\psi_\alpha(\beta) = \ot(\sigma_\alpha[\beta])$ for all $\beta < \kappa$.

\begin{lemma}
\label{wps-equivs}
Suppose $\calI$ is a normal ideal on a successor cardinal $\kappa$.  The following are equivalent:
\begin{enumerate}
\item\label{wps-genult} $\calI$ is weakly presaturated.
\item\label{wps-dense} For all $\calI$-positive $A \subseteq \kappa$ and all $f : \kappa \to \kappa$, there is $\alpha<\kappa^+$ and an $\calI$-positive $B \subseteq A$ such that $f(\beta) < \psi_\alpha(\beta)$ for all $\beta \in B$.
\item\label{otdom} There is $\theta > \kappa$ and a stationary set $\calS \subseteq \{ M \prec H_\theta : M \cap \kappa \in \kappa \}$ such that:
\begin{enumerate}
\item $\calI$ is the canonical projection of $\ns \rest \calS$.
\item For all $A \in \calI^+$ and all $f : \kappa \to \kappa$, the set
$$\{ M \in \calS : M \cap \kappa \in A \wedge f(M \cap \kappa) < \ot(M \cap \kappa^+)  \}$$
is stationary.
\end{enumerate}
\end{enumerate}
\end{lemma}

\begin{proof}
Suppose (\ref{wps-genult}).  Let $A$ be $\calI$-positive and $f : \kappa \to \kappa$.  $f$ represents some ordinal $<j(\kappa)$, and thus there is some $\calI$-positive $B \subseteq A$ deciding that for some canonical function $\psi_\alpha$, $[f]_{\dot G} < [\psi_\alpha]_{\dot G}$.  This implies that for $\calI$-almost-all $\beta \in B$, $f(\beta) <\psi_\alpha(\beta)$, establishing (\ref{wps-dense}).  On the other hand, if (\ref{wps-dense}) holds, we have that every ordinal of the generic ultrapower $<j(\kappa)$ is represented by a function $f : \kappa \to \kappa$, and the statement says that for such $f$, there is a dense set of conditions in the forcing $\p(\kappa)/\calI$ deciding that $f$ represents an ordinal below $(\kappa^+)^V$, showing (\ref{wps-genult}).

Suppose (\ref{wps-dense}).  By Foreman's theorem, $\calI$ is the canonical projection of $\ns \rest \calG$, where $\calG$ is the set of $\calI$-good $M \prec H_{(2^{\kappa})^+}$.  Suppose $f : \kappa \to \kappa$ and $A$ is $\calI$-positive.  By hypothesis, there is an $\calI$-positive $B \subseteq A$ and an $\alpha<\kappa^+$ such that $\psi_\alpha$ dominates $f$ on $B$.  There are stationary-many $M \in \calG$ such that $M \cap \kappa \in B$.  
Let $\sigma_\alpha : \kappa \to \alpha$ be a surjection defining $\psi_\alpha$.  
If $M$ is any such model with $\sigma_\alpha \in M$, then 
$$\psi_\alpha(M \cap \kappa) = \ot(\sigma_\alpha[M \cap \kappa]) = \ot(M \cap \alpha) < \ot(M \cap \kappa^+).$$
Thus there are stationary-many $M \in \calG$ with $M \cap \kappa \in A$ and $f(M \cap \kappa) < \psi_\alpha(M \cap \kappa) < \ot(M \cap \kappa^+)$, establishing (\ref{otdom}).

Suppose (\ref{otdom}).  Let $A$ be $\calI$-positive and $f : \kappa \to \kappa$.  There are stationary-many $M \in \calS$ such that $M\cap \kappa \in A$ and $f(M \cap \kappa) < \ot(M \cap \kappa^+)$.  By Fodor's Lemma, there is an $\alpha<\kappa^+$ and a stationary $\calT \subseteq \calS$ such that for all $M \in \calT$, $\alpha \in M$, $M \cap \kappa \in A$, and $f(M \cap \kappa) < \ot(M \cap \alpha)$.  We may assume that there is a surjection $\sigma_\alpha : \kappa \to \alpha$ that belongs to each $M \in \calT$.  For all $M \in \calT$, $f(M \cap \kappa) < \ot(M \cap \alpha) = \ot(\sigma_\alpha[M \cap \kappa]) = \psi_\alpha(M \cap \kappa)$.  Let $B = \{ M \cap \kappa : M \in \calT \}$.  Then $B$ is an $\calI$-positive subset of $A$, and $\psi_\alpha$ dominates $f$ on $B$.  This shows (\ref{wps-dense}).
\end{proof}

Chang's Conjecture $(\kappa_1,\kappa_0) \chang (\mu_1,\mu_0)$ states that for every function $F : \kappa_1^{<\omega} \to \kappa_1$, there is a $z \subseteq \kappa_1$ closed under $F$ such that $|z| = \mu_1$ and $|z \cap \kappa_0| = \mu_0$.  If $\kappa = \mu^+$, then we abbreviate $(\kappa^+,\kappa) \chang (\mu^+,\mu)$ by $\cc(\kappa)$, which is equivalent to the statement that for all $F : (\kappa^+)^{<\omega} \to \kappa^+$, there is $z \subseteq \kappa^+$ such that $z \cap \kappa \in \kappa$ and $\ot(z) = \kappa$ (see \cite{foreman}). 
 According to \cite{dl}, Magidor showed that if $\kappa = \mu^+$ and $\cc(\kappa)$ holds, then there is a weakly presaturated ideal on $\kappa$. (Note that the dual of a weakly presaturated ideal is called a ``flat filter'' there.)  In fact, such an ideal is obtained as the canonical projection of the nonstationary ideal on $\p(\kappa^+)$ restricted to the set $\{ z : z \cap \kappa \in \kappa \wedge \ot(z) = \kappa \}$.  A result of Foreman and Magidor \cite{fm} shows that all but nonstationary-many $z$ in this set satisfy $\cf(z \cap \kappa) = \cf(\mu)$.

The \emph{weak Chang's Conjecture} at a regular cardinal $\kappa$ says: for all functions $F : (\kappa^+)^{<\omega} \to \kappa^+$, there is $\alpha<\kappa$ such that for all $\beta<\kappa$, there is $z \subseteq \kappa^+$ closed under $F$ with $z \cap \kappa = \alpha$ and $\ot(z) >\beta$.
  To further specify this kind of property, let us write $\wcc(\kappa,A)$ for $A \subseteq \kappa$ to assert that such $\alpha$ can be found in $A$.  
  When $\kappa = \mu^+$, the result of Foreman and Magidor mentioned above tells us that $\cc(\kappa)$ implies $\wcc(\kappa,\cof(\mu))$.
  The next result shows that $\wcc(\kappa,A)$ follows from the existence of a weakly presaturated normal ideal $\calI$ on $\kappa$ such that $A \in \calI^d$.

\begin{lemma}
Suppose $\kappa$ is a regular cardinal and $A \subseteq \kappa$.  The following are equivalent:
\begin{enumerate}
\item\label{wcc} $\wcc(\kappa,A)$.
\item\label{wccmod} For all $f : \kappa \to \kappa$, there are stationary-many $M \in \p_\kappa(H_{\kappa^+})$ such that $M \cap \kappa \in A$ and $f(M\cap\kappa) < \ot(M\cap\kappa^+)$.
\item\label{canfunnodom} For all $f : \kappa \to \kappa$, there is $\alpha<\kappa^+$ such that $\{ \beta \in A : f(\beta) < \psi_\alpha(\beta) \}$ is stationary.
\end{enumerate}
\end{lemma}

\begin{proof}
Suppose $\wcc(\kappa,A)$ holds and $f : \kappa \to \kappa$.  Let $\frak A$ be any structure on $H_{\kappa^+}$ in a countable language with a well-order.
If $F : (\kappa^+)^{<\omega}\to\kappa^+$ is the result of restricting inputs and outputs of Skolem functions for $\frak A$ to ordinals, then every $z$ closed under $F$ has $\sk^{\frak A}(z) \cap\kappa^+ = z$.  By $\wcc(\kappa,A)$, there is $M \prec \frak A$ such that $M \cap 
\kappa \in A$ and $f(M \cap \kappa)<\ot(M \cap \kappa^+)$.  Thus (\ref{wcc}) implies (\ref{wccmod}).

Suppose towards a contradiction that  (\ref{wccmod}) holds and (\ref{canfunnodom}) fails.  Let $\lhd$ be a well-order of $H_{\kappa^+}$, and assume that each canonical function $\psi_\alpha$ on $\kappa$ is defined using the $\lhd$-least surjection $\sigma : \kappa \to \alpha$.  Let $f : \kappa \to \kappa$ be such that for all $\alpha<\kappa^+$, $\psi_\alpha(\beta) < f(\beta)$ for all but nonstationary-many $\beta \in A$.  Let $\frak A = \la H_{\kappa^+},\in,\lhd,f,A \ra$.  Let $M \prec \frak A$ be such that $M \cap \kappa \in A$ and $f(M \cap \kappa)<\ot(M \cap \kappa^+)$.  Let $\alpha \in M$ be such that $f(M \cap \kappa)<\ot(M \cap \alpha)$.  Let $\sigma : \kappa \to \alpha$ be the $\lhd$-least surjection, so $\sigma \in M$ and $\psi_\alpha(\beta) = \ot(\sigma[\beta])$ for all $\beta<\kappa$. By elementarity, there is a club $C \subseteq \kappa$ in $M$ such that $\psi_\alpha(\beta) < f(\beta)$ for all $\beta \in C\cap A$.  But clearly $M \cap \kappa \in C \cap A$ and $\psi_\alpha(M \cap \kappa) = \ot(M \cap \alpha)>f(M \cap \kappa)$, a contradiction.

Suppose towards a contradiction that (\ref{canfunnodom}) holds but $\wcc(\kappa,A)$ fails.  Let $F : (\kappa^+)^{<\omega} \to \kappa^+$ be such that for all $\alpha \in A$, the set $\{ \ot(z) : z \cap \kappa = \alpha \wedge F[z^{<\omega}] \subseteq z \}$ is bounded below $\kappa$.  Let $f : A \to \kappa$ return such a bound.  Let $\frak A$ be a structure on $H_{\kappa^+}$ in a countable language such that if $M \prec \frak A$, then $M \cap \kappa^+$ is closed under $F$.  By (\ref{canfunnodom}), let $\alpha<\kappa^+$ be such that $B = \{ \beta \in A : f(\beta) < \psi_\alpha(\beta) \}$ is stationary.  Let $M \prec \frak A$ be such that $M \cap \kappa \in B$ and $\alpha,f \in M$.  Then $f(M \cap \kappa)<\psi_\alpha(M \cap \kappa) = \ot(M \cap \alpha)$, contrary to the definition of $f$.
\end{proof}

\begin{corollary}
\label{wps-wcc}
Suppose $\kappa$ is a regular cardinal and $A \subseteq \kappa$ is stationary.  The following are equivalent.
\begin{enumerate}
\item $\ns_\kappa \rest A$ is weakly presaturated.
\item $\wcc(\kappa,B)$ holds for every stationary $B \subseteq A$.
\end{enumerate}
\end{corollary}

\section{Chang's Conjecture and coherent sequences}
\label{sec:ccsquare}

In this section, we discuss several ways in which versions of Chang's Conjecture can lead to non-compact structures, expanding on the theme of \cite{coxeskew}.  Let us first recall some basic definitions and facts.  For a cardinal $\kappa$, Jensen's \emph{weak square principle} $\square^*_\kappa$ says that there is a sequence $\la \calC_\alpha : \alpha < \kappa^+ \ra$ such that each $\calC_\alpha$ is a set of at most $\kappa$-many club subsets of $\alpha$, each of order-type at most $\kappa$, and if $C \in \calC_\alpha$ and $\beta$ is a limit point of $C$, then $C \cap \beta \in \calC_\beta$.

These principles are closely connected to properties of trees.
If $\kappa$ is a regular cardinal, $T$ is a \emph{$\kappa$-tree} when it is a partially ordered set, well-ordered below any given node, such that the rank of any given node is ${<}\kappa$, and if $T_\alpha$ is the set of nodes of rank $\alpha$, then $|T_\alpha|<\kappa$.  A $\kappa$-tree is \emph{Aronszajn} if it has no cofinal branch, i.e.\ a linearly ordered subset of order-type $\kappa$.  If $\kappa = \mu^+$, a $\kappa$-tree $T$ is \emph{special} if there is a function $f : T \to \mu$ such that if $a <_T b$, then $f(a) \not= f(b)$.  Clearly, special trees are Aronszajn.  Jensen \cite{jensen} showed that $\square^*_\kappa$ is equivalent to the existence of a special $\kappa^+$-tree.  %For a proof, see \cite{cummings}.

\begin{theorem}
\label{wccsquare}
Suppose $\mu$ is a regular cardinal, $2^{<\mu} \leq \mu^+$, and $\wcc(\mu^+,\cof(\mu))$.  Then $\square^*_\mu$ holds.
\end{theorem}

\begin{proof}
We can assume that $\mu$ is uncountable, since $\square^*_\omega$ always holds.
Let $\kappa = \mu^+$, let $\lhd$ be a well-order of $H_{\kappa^+}$, and let $\frak A = (H_{\kappa^+},\in,\lhd)$.
By $\wcc(\kappa,\cof(\mu))$, let $\delta<\kappa$ be such that $\cf(\delta) = \mu$ and for every $\alpha < \kappa$, there is $M \prec \frak A$ such that $M \cap \kappa = \delta$ and $\ot(M \cap \kappa^+) > \alpha$.
Let $\la x_\alpha : \alpha < \kappa \ra$ be the $\lhd$-least enumeration of $[\kappa]^{<\mu}$.  Let $Q_0 = \{ x_\alpha : \alpha < \delta \}$.  For each $M \prec \frak A$ with $M \cap \kappa = \delta$, $M \cap [\kappa]^{<\mu} = Q_0 \subseteq [\delta]^{<\mu}$.

\begin{claim}
$Q_0$ is $\subseteq$-cofinal in $[\delta]^{<\mu}$.
\end{claim}

\begin{proof}[Proof of claim]
For $\alpha < \kappa$ let $\sigma_\alpha$ be the $\lhd$-least surjection from $\mu$ to $\alpha$.  Then for each $\alpha<\delta$ and each $M \prec \frak A$ with $M \cap \kappa \geq \delta$, $\sigma_\alpha \in M$.
If $x \in [\delta]^{<\mu}$, then since $\cf(\delta) = \mu$ and $|x|<\mu$, there is $\alpha<\delta$ such that $x \subseteq \alpha$, and there is $\beta<\mu$ such that $x \subseteq \sigma_\alpha[\beta] \in Q_0$.
\end{proof}

Let $\calE$ be the set of transitive structures $N \in H_\kappa$ such that $N$ is elementarily equivalent to $H_{\kappa^+}$, $N$ is correct about cardinals $\leq \mu$, $\delta = (\mu^+)^N$, $\delta$ is the largest cardinal in $N$, and $N \cap [\delta]^{<\mu} = Q_0$.
By taking transitive collapses of $M \prec \frak A$ with $M \cap \kappa = \delta$, we have that for every $\alpha<\kappa$, there is an $N \in \calE$ with $\alpha \in N$.
%transitive structure $N \in H_\kappa$ such that $N$ is elementarily equivalent to $H_{\kappa^+}$, $N$ is correct about cardinals $\leq \mu$, $\delta = (\mu^+)^N$, $\delta$ is the largest cardinal in $N$, $N \cap [\delta]^{<\mu} = Q_0$, and $\alpha \in N$.  Let us call the set of such structures $\calE$.

Suppose $N \in \calE$ and $\alpha \in N \cap \kappa$.  Then there is a surjection $f : \delta \to \alpha$ in $N$, and $f$ is coded by a set $X \subseteq \delta$ in $N$.  More specifically, $X$ codes, via G\"odel pairing, a prewellordering of $\delta$ of length $\alpha$, which is inter-constructible with the surjection $f$.  $X$ has the property that whenever $z \in Q_0$, then $X \cap z \in Q_0$, since $X \cap z \in ([\delta]^{<\mu})^N = Q_0$.  Following the terminology of \cite{coxeskew}, we say  that \emph{$X$ is approximated by $Q_0$}.  
%Thus for every $\alpha<\kappa$, there is an $X \subseteq \delta$ such that $X$ codes a prewellordering of $\delta$ of length $\alpha$, and $X$ is approximated by $Q_0$.

Let $N^* = \sk^{\frak A}(\delta +1)$.
Let $Q_1 =  [\delta]^{<\mu} \cap N^* \supseteq Q_0$.
Let $C^* \in N^*$ be a club in $\delta$ of order-type $\mu$.  We may assume that only limit ordinals are in $C^*$.

\begin{claim}
\label{canonical}
Suppose $\alpha <\kappa$, and $X,Y$ are two subsets of $\delta$ that are approximated by $Q_0$ and code prewellorderings of $\delta$ of length $\alpha$.  Let $f_X,f_Y$ be the corresponding surjections from $\delta$ to $\alpha$.  Then:
$$\{ f_X[z] : z \in Q_1 \} = \{ f_Y[z] : z \in Q_1 \}$$
\end{claim}

\begin{proof}[Proof of claim]
Let $r \in Q_1$ be arbitrary.  We must find $s \in Q_1$ such that $f_X[r] = f_Y[s]$.
Let $\sigma : \mu \to \delta$ be a surjection in $N^*$. %and let $z_\beta = \sigma[\beta]$, which is in $Q_1$.  
There is a club $C \subseteq \mu$ such that for all $\beta \in C$, $f_X \circ \sigma[\beta] = f_Y \circ \sigma[\beta]$ and $\sigma[\beta]$ is closed under G\"odel pairing.  Let $\gamma \in C$ be such that $r \subseteq \sigma[\gamma] = z$.  Let $z' \in Q_0$ be such that $z \subseteq z'$.
Then $X \cap z'$ and $Y \cap z'$ are in $Q_0$, and thus $X \cap z = (X\cap z')\cap z$ and $Y \cap z = (Y\cap z')\cap z$ are in $Q_1$.  
$X \cap z$ and $Y \cap z$ code prewellorderings of $z$ of length $\eta = \ot(f_X[z]) = \ot(f_Y[z]) <\mu$.  Let $g_X$ and $g_Y$ be the corresponding surjections from $z$ to $\eta$.  %They are both members of $N^*$.

Let $r' = g_X[r]$.  Note that if $\pi : \eta \to f_X[z]$ is the unique order-preserving map, then $\pi[r'] = f_X[r]$.  Let $s = g_Y^{-1}[r']$.  Then $s \in Q_1$, and $g_Y[s] = r'$.  Moreover, $f_Y[s] = \pi \circ g_Y[s] = f_X[r]$.
\end{proof}

For $\alpha <\kappa$ of cofinality $<\mu$, let 
$$\calC_\alpha = \{ f[z] : z \in Q_1 \wedge f[z] \text{ is club in } \alpha \},$$
where $f$ is a surjection from $\delta$ to $\alpha$ in some $N \in \calE$.  By Claim \ref{canonical}, this set is the same no matter which such $f$ we choose.  Note that $\calC_\alpha$ has size at most $|Q_1| = \mu$.  Furthermore, this set is nonempty for each such $\alpha$.  To see this, it suffices to show that whenever $\alpha \in N \in \calE$, then $N \models \cf(\alpha)<\mu$.  If this were to fail, then such an $N$ would have an increasing cofinal sequence in $\alpha$ of length either $\mu$ or $\delta$.  But thus is impossible, since $\mu$ is regular and $\cf(\delta) = \mu$.

For $\alpha<\kappa$ of cofinality $\mu$, there are two cases.  In the first case, there is some $N \in \calE$ with $\alpha \in N$ such that $N \models \cf(\alpha) = \mu$.  Choose such a model $N_\alpha$ and a club $C_\alpha \subseteq \alpha$ in $N_\alpha$ with $\ot(C_\alpha) = \mu$, and let $\calC_\alpha = \{ C_\alpha \}$.  In the second case, all $N \in \calE$ with $\alpha \in N$ satisfy that $\cf(\alpha) = \delta$.  Choose such an $N_\alpha$ and choose a club $D_\alpha \subseteq \alpha$ in $N_\alpha$ such that $\ot(D_\alpha) = \delta$.  Let $\la \xi_i : i < \delta \ra$ be its increasing enumeration.  Then let $C_\alpha = \{ \xi_i : i \in C^* \}$, and let $\calC_\alpha = \{ C_\alpha \}$.

We want to show that the sequence $\la \calC_\alpha : \alpha < \kappa \ra$ is coherent.  Suppose $\cf(\alpha)<\mu$, $C \in \calC_\alpha$, and $\beta$ is a limit point of $C$.  Let $N \in \calE$ be such that $\alpha \in N$, and let $f : \delta \to \alpha$ be a surjection in $N$ with the property that $C = f[z]$ for some $z \in Q_1$.  Let $g : \delta \to \beta$ be the surjection in $N$ defined by putting $g(\gamma) = f(\gamma)$ when $f(\gamma)<\beta$, and otherwise letting $g(\gamma)$ be the least point of $C$.  Then $C \cap \beta = g[z] \in \calC_\beta$.

Suppose that $\alpha<\kappa$ and some $N \in \calE$ satisfies that $\cf(\alpha) = \mu$.  Let $N_\alpha,C_\alpha$ be as in the construction of $\calC_\alpha$.  Let $f : \delta \to \alpha$ be a surjection in $N_\alpha$ with the property that $f \rest \mu$ enumerates $C_\alpha$ in increasing order.  Then for any limit point $\beta$ of $C_\alpha$, there is $\gamma<\mu$ such that $C_\alpha \cap \beta = f[\gamma]$.  If $g : \delta \to \beta$ is a surjection in $N_\alpha$ defined from $f$ as in the previous case, then $g[\gamma] = C_\alpha \cap \beta$, and of course $\gamma \in Q_1$, so $C_\alpha \cap \beta \in \calC_\beta$.

Suppose that $\alpha<\kappa$ and all $N \in \calE$ with $\alpha \in N$ satisfy that $\cf(\alpha) = \delta$.  Let $N_\alpha,D_\alpha,C_\alpha$ be as in the construction of $\calC_\alpha$.  Let $f : \delta \to \alpha$ be a surjection in $N_\alpha$ with the property that $f$ restricted to the limit ordinals enumerates $D_\alpha$ in increasing order.  Let $\beta$ be a limit point of $C_\alpha$.  Then there is some $\gamma < \delta$ such that $f[C^* \cap \gamma] = C_\alpha \cap \beta$.  If $g : \delta \to \beta$ is a surjection in $N_\alpha$ defined from $f$ as in the previous cases, then $g[C^* \cap \gamma] = C_\alpha \cap \beta$.  Since $C^* \cap \gamma \in Q_1$, $C_\alpha \cap \beta \in \calC_\beta$.
\end{proof}

Suppose $\mu$ is a regular cardinal, $\kappa = \mu^+$, $M \prec H_{\kappa^+}$, $M \cap \kappa = \delta < \kappa$, and $\cf(\delta) = \mu$.  For any $\alpha \in M \cap \kappa^+$, if $\alpha \not= \sup(M \cap \alpha)$, then $\alpha$ must have cofinality $\kappa$, since otherwise $\cf(\alpha) \subseteq M$, and so there would be a cofinal subset of $\alpha$ of order-type $\cf(\alpha)$ that is both an element and a subset of $M$.  If $f : \kappa \to \alpha$ is an increasing cofinal function in $M$, then $f \rest \delta$ is cofinal in $M \cap \alpha$, and so $\cf(M \cap \alpha) = \mu$.  Thus if $g : \ot(M\cap\kappa^+)\to M\cap\kappa^+$ is the increasing enumeration, then the discontinuity points of $g$ all have cofinality $\mu$.

If $\cc(\kappa)$ holds, then stationary-many $M \prec H_{\kappa^+}$ have $M \cap \kappa = \delta <\kappa$ with $\cf(\delta) =\mu$, and $\ot(M \cap \kappa^+) = \kappa$.  Thus $M \cap \kappa^+$ is a ${<}\mu$-closed set of ordinals.  Such $M$ will have some discontinuity points of cofinality $\mu$, but it is natural to ask whether there might be a club in $\sup(M \cap \kappa^+)$ that avoids them.  Let $\ccc(\kappa)$ stand for the assertion that there are stationary-many such models.

\begin{prop}
\label{ccc}
For regular $\mu$,
$\ccc(\mu^+)$ implies $\square^*_\mu$.
\end{prop}

\begin{proof}
Let $\kappa = \mu^+$.
It follows from a result of Shelah \cite{shelah91} that $\kappa^+$ carries a ``partial weak sqaure,'' i.e.\ a sequence $\vec\calC = \la \calC_\alpha : \alpha \in \kappa^+ \cap \cof({<}\kappa) \ra$ such that each $\calC_\alpha$ is a set of at most $\kappa$-many clubs in $\alpha$, each of order-type $\leq\mu$, and if $C \in \calC_\alpha$ and $\beta \in \lim C$, then $C \cap \beta \in \calC_\beta$.
Let $\frak A = \la H_{\kappa^+},\in,\vec\calC \ra$.  Let $M \prec \frak A$ be such that $\ot(M \cap \kappa^+) = \kappa$, $M \cap \kappa = \delta < \kappa$, and $M \cap \kappa^+$ contains a club $D \subseteq \sup(M \cap \kappa^+)$.  We may assume that $D$ possesses only ordinals of cofinality ${<}\kappa$.
%By the result of Foreman-Magidor \cite{fm}, $\cf(\delta) = \mu$, and so by the above observations, $M \cap \kappa^+$ is ${<}\mu$-closed.  
Let $\pi : M \to N$ be the transitive collapse.  Let $\bar x = \pi(x)$ for each $x \in M$, and let $\bar D = \pi[D]$.

For each $\alpha \in \bar D$, $\bar{\calC}_\alpha$ is defined, and $\pi(C) = \pi[C]$ for each $C \in \calC_\alpha$.  Thus $\bar{\calC}_\alpha$ is a set of at most $\mu$-many clubs in $\alpha$, each of order-type at most $\mu$.  To complete this to a $\square^*_\mu$-sequence, we only need to fill in the gaps at ordinals in $\kappa \setminus \bar D$.  But it is a standard fact, easy to show by induction, that for each $\eta<\kappa$, there exists a ``short square sequence'' of length $\eta$, i.e.\ a sequence $\la C_\alpha : \alpha < \eta \ra$ such that each $C_\alpha$ is a club in $\alpha$ of order-type at most $\mu$, and if $\beta \in \lim C_\alpha$, then $C_\alpha \cap \beta = C_\beta$.  Thus for $\alpha \in \bar D$, choose a short square sequence $\la C_\beta : \alpha < \beta < \alpha' \ra$, where $\alpha'$ is the next point of $\bar D$ above $\alpha$, and each $C_\beta$ is contained in $[\alpha,\beta)$.  Then putting $\calD_\alpha = \bar\calC_\alpha$ for $\alpha \in \bar D$ and $\calD_\alpha = \{ C_\alpha \}$ for $\alpha \in \kappa \setminus \bar D$, we have that $\la \calD_\alpha : \alpha < \kappa \ra$ is a $\square^*_\mu$-sequence.
\end{proof}

A weakening of the principle $\square^*_\kappa$ is the assertion that $\kappa^+$ is \emph{approachable}, a notion due to Shelah \cite{shelah78}.  More generally, we say that an ordinal $\alpha<\kappa^+$ is \emph{approachable with respect to a sequence} $\la x_\beta : \beta < \kappa^+ \ra$ if there is a cofinal set $y \subseteq \alpha$ of order-type $\cf(\alpha)$ such that all initial segments are in $\{ x_\beta : \beta<\alpha \}$, and we say that
 a set $S \subseteq \kappa^+$ is \emph{approachable} when there is a sequence $\la x_\beta : \beta < \kappa^+ \ra$ and a club $C \subseteq \kappa$ such that all $\alpha \in S \cap C$ are approachable with respect to $\la x_\beta : \beta < \kappa^+ \ra$.  Shelah's partial square result \cite{shelah91} shows that if $\kappa$ is regular, then $\kappa^+ \cap \cof({<}\kappa)$ is approachable.  Moreover, the collection of approachable subsets of $\kappa^+$ forms a normal ideal denoted by $I[\kappa^+]$.  We say that $I[\kappa^+]$ is \emph{trivial} when it is as small as possible under the constraint imposed by Shelah's partial square result, namely $I[\kappa^+] = \ns_{\kappa^+} \rest \cof(\kappa)$.  If $2^{<\kappa} \leq \kappa^+$, then it is not hard to show using an enumeration of $[\kappa^+]^{<\kappa}$ that $I[\kappa^+]$ is nontrivial, so the triviality of $I[\kappa^+]$ requires $2^{<\kappa} > \kappa^+$.  Mitchell \cite{mitchellapp} showed that it is consistent relative to a greatly Mahlo cardinal that $I[\omega_2]$ is trivial.

\begin{prop}
If $\mu$ is regular and uncountable and $\cc(\mu^+)$ holds, then $I[\mu^{++}]$ is nontrivial.
\end{prop}

\begin{proof}
Let $\kappa = \mu^+$.  Let $\frak A = \la H_{\kappa^+},\in,\lhd \ra$, where $\lhd$ is a well-order of $H_{\kappa^+}$.  Suppose $M$ and $N$ are elementary in $\frak A$, $M \cap \kappa = N \cap \kappa = \delta<\kappa$, $\sup(M \cap \kappa^+) = \sup(N \cap \kappa^+) = \eta$, and $\cf(\delta),\cf(\eta) \geq \mu$.  By the observations preceding Proposition \ref{ccc}, both $M \cap \kappa^+$ and $N \cap \kappa^+$ are ${<}\mu$-closed.  Then $M \cap N \cap \eta$ is an unbounded subset of $\eta$.  If $\alpha \in M \cap N \cap \eta$, then there is a surjection $f : \kappa \to \alpha$ in $M \cap N$.  We have that $f[\delta] = M \cap \alpha = N \cap \alpha$.  As $\alpha$ was arbitrary, it follows that $M \cap \kappa^+ = N \cap \kappa^+$.  Thus for each $\delta<\kappa$ and $\eta<\kappa^+$, each of cofinality $\geq \mu$, there is at most one set $X(\delta,\eta)$ that equals $M \cap \kappa^+$ for some $M \prec \frak A$ with $M \cap \kappa = \delta$ and $\sup(M \cap \kappa^+) = \eta$.  Let $\la x_\alpha : \alpha < \kappa^+ \ra$ enumerate all initial segments of such $X(\delta,\eta)$.

Now let $C \subseteq \kappa^+$ be a club.  Let $C' \subseteq C$ be a club such that whenever $\alpha \in C'$, $\delta,\eta<\alpha$, and $X(\delta,\eta)$ is defined, then $X(\delta,\eta)$ and all its initial segments are in $\{ x_\beta : \beta < \alpha \}$.  Expand $\frak A$ to $\frak A'$ to include a constant for $C'$.  By $\cc(\kappa)$, let $M \prec \frak A'$ be such that $M \cap \kappa = \delta < \kappa$ and $\ot(M \cap \kappa^+) = \kappa$.  Then $\sup(M \cap \kappa^+) \in C'$.   By Foreman-Magidor \cite{fm}, $\cf(\delta) =\mu$.   For all $\eta \in M \cap \kappa^+$ of cofinality $\geq \mu$ such that $\sk^{\frak A}(M \cap \eta) \cap \eta = M \cap \eta$, we have $M \cap \eta = X(\delta,\eta)$.  Thus $M \cap \kappa^+$ has the property that all initial segments are in $\{ x_\alpha : \alpha < \sup(M \cap \kappa^+) \}$.  This shows that there is an ordinal in $C$ that is approachable with respect to the sequence $\la x_\alpha : \alpha < \kappa^+ \ra$.
\end{proof}

The above result cannot be strengthened much further, since in general, $\cc(\kappa)$ does not imply that the full $\kappa^+$ is approachable when $\kappa$ is the successor of a regular cardinal.  We give a sketch of the consistency proof for the reader who is familiar with both Mitchell's forcing for the tree property \cite{mitchell} and the Kunen-style forcings for saturated ideals and Chang's Conjecture \cite{kunen}.

\begin{prop}
If $\kappa$ is huge and $\mu<\kappa$ is regular, then there is a forcing extension in which $\kappa = \mu^+$, $\cc(\kappa)$ holds, and $\kappa^+$ is not approachable.
\end{prop}

\begin{proof}[Proof (sketch)]
We first define a variant of Mitchell's forcing using Easton supports.  If $\mu$ is a regular cardinal and $\kappa > \mu$ is inaccessible, then the standard Mitchell forcing to make $\kappa = \mu^{++}$ consists of pairs $\la p,q \ra$, where:
\begin{enumerate}
\item $p \in \add(\mu,\kappa)$.
\item\label{supp} $q$ is a function with domain in $[\kappa]^{<\mu^+}$.
\item for each $\alpha \in \dom(q)$, $q(\alpha)$ is a name for a condition in $\col(\mu^+,\alpha)$, as defined in $V^{\add(\mu,\alpha)}$.
\end{enumerate}
We put $\la p_1,q_1 \ra \leq \la p_0,q_0 \ra$ when $p_1 \leq p_0$ and for each $\alpha\in\dom(q)$, $p_1 \rest \alpha \Vdash q_1(\alpha) \leq q_0(\alpha)$.  Variations on this forcing fix some $A \subseteq \kappa$ and require the domain of the second coordinate to be in $[A]^{<\mu^+}$.  This can change the combinatorial effects of the forcing, such as whether $\kappa$ is forced to be approachable or not (see \cite{8fold}).  In particular, if $\kappa$ is Mahlo and $A$ is the set of inaccessibles below $\kappa$ which are not limits of inaccessibles, then $\kappa \setminus A$ is forced to be non-approachable.  

Now, our modification is simply to require that the second coordinate is function with domain an Easton subset of $A$, rather than a ${\leq}\mu$-sized subset of $A$.  Recall that a set of ordinals $X$ is \emph{Easton} when $\sup(X \cap \alpha) < \alpha$ whenever $\alpha$ is regular.  It is not hard to check that the same arguments of \cite{8fold} work to show that this modification still forces that $\kappa$ is not approachable.  Let us call this forcing $\mathbb P(\mu,\kappa)$.

It is a standard fact, owing in part to the supports of the functions in the second coordinate, that Mitchell's forcing is a projection of $\add(\mu,\kappa) \times \col(\mu^+,{<}\kappa)$.  The same analysis yields that our modified poset is a projection of $\add(\mu,\kappa) \times \mathbb E(\mu^+,\kappa)$, where $\mathbb E(\mu^+,\kappa)$ is the Easton-support product of $\col(\mu^+,\alpha)$ over $\alpha<\kappa$, or the Easton collapse, introduced by Shioya \cite{shioya}.

By the work of the author and Hayut \cite{eh}, if $\kappa$ is huge with target $\theta$ and $\mu<\kappa$ is regular, then the two-step iteration of Easton collapses, $\mathbb E(\mu,\kappa) * \dot{\mathbb E}(\kappa,\theta)$, forces $(\theta,\kappa) \chang (\kappa,\mu)$.  The reason is that, if $G * H \subseteq \mathbb E(\mu,\kappa) * \dot{\mathbb E}(\kappa,\theta)$ is generic, then a hugeness embedding $j : V \to M$, with $\crit(j) =\kappa$, $j(\kappa) = \theta$, and $j[\theta] \in M$, can be lifted by a further forcing, yielding an embedding $j : V[G][H] \to M[G'][H']$.  Then for any $F : [\theta]^{<\omega} \to \theta$ in $V[G][H]$, $j[\theta]$ is a set closed under $j(F)$ with $j[\theta] \cap j(\kappa) = \kappa < j(\kappa)$ and $\ot(j[\theta]) = \theta = j(\kappa)$.  By elementarity, there is a set $X \in V[G][H]$ closed under $F$ with $X \cap \kappa \in \kappa$ and $\ot(X) = \kappa$.  Since $V[G][H] \models \mu^{<\mu} = \mu$, $\add(\mu,\kappa)$ is $\kappa$-c.c.\ in this model.  Thus by standard arguments, we can lift the embedding further through any $K \subseteq \add(\mu,\kappa)$ that is generic over $V[G][H]$, and $V[G][H][K]$ will satisfy $(\theta,\kappa) \chang (\kappa,\mu)$ for the same reason.

Let $G * Q \subseteq \mathbb E(\mu,\kappa) * \dot{\mathbb P}(\mu,\theta)$ be generic.  In $V[G][Q]$, $\kappa = \mu^+$, $\theta = \kappa^+$, and $\theta$ is not approachable.  Then force with the quotient $(\add(\mu,\theta) \times \mathbb E(\kappa,\theta)) / Q$, where the product is as defined in $V[G]$, yielding a generic $G * (H \times K)$ for $\mathbb E(\mu,\kappa) * (\dot{\mathbb E}(\kappa,\theta) \times \dot\add(\mu,\theta))$.  A further forcing yields a lifted embedding $j : V[G][H][K] \to M[G'][H'][K']$ with $\crit(j) = \kappa$, $j(\kappa) = \theta$, and $j[\theta] \in M$.  The generic $H' \times K'$ projects to a generic $Q'$ for $j(\mathbb P(\mu,\theta))$, and restricting the map yields an elementary $j : V[G][Q] \to M[G'][Q']$.  The same reflection argument as above then gives that $V[G][Q] \models (\theta,\kappa) \chang (\kappa,\mu)$.
\end{proof}

\section{Modified Neeman forcing}
\label{sec:neeman}

Let us recall the definition of Neeman's model sequence poset \cite{neeman}.  We fix some transitive set $K$ satisfying a sufficient amount of ZFC.  In our applications, $K$ will always be $H_\theta$ for some regular $\theta$.  We fix two classes of elementary submodels of $K$, the small models $\calS$ and the transitive models $\calT$, and a cardinal $\kappa$ with the following properties:
\begin{itemize}
\item $\calS \cup \calT \subseteq K$. %, and $K = \bigcup\calS = \bigcup\calT$.
\item $\kappa+1 \subseteq M$ for all $M \in \calS \cup \calT$.
\item If $M,N \in \calS$ and $M \in N$, then $M \subseteq N$.
\item If $W \in \calT$, $M \in \calS$, and $W \in M$, then $M \cap W \in \calS \cap W$.
\item Each $W \in \calT$ is transitive and ${<}\kappa$-closed.
\end{itemize}
A pair $\la \calS,\calT \ra$ satisfying these conditions is called \emph{appropriate} for $\kappa$ and $K$.
Usually these conditions are implied by defining $\calS$ and $\calT$ such that for some regular cardinal $\lambda\geq\kappa$, every $M \in \calS$ has $M \cap \lambda \in \lambda$ and $|M|<\lambda$, while every $W \in \calT$ is ${<}\lambda$-closed.
The poset $\mathbb P_{\kappa,\calS,\calT}$ consists of sets of models $s \in [\calS \cup \calT]^{<\kappa}$ such that the rank function is injective on $s$, and if $ \la M_i : i < \alpha \ra$ enumerates $s$ by order of rank, then for each $\beta,\gamma<\alpha$:
\begin{itemize}
\item $\{ i < \beta : M_i \in M_\beta \}$ is cofinal in $\beta$.
\item $s \cap M_\beta \in M_\beta$.
\item $M_\beta \cap M_\gamma = M_i$ for some $i<\alpha$.
\end{itemize}
A condition $t$ is stronger than a condition $s$ when $t \supseteq s$.
Neeman also introduced a \emph{decorated} version of this poset, which enforces some continuity of the generic object added.  We will slightly modify these decorations as follows.  We define $\mathbb P^\mathrm{dec*}_{\kappa,\calS,\calT}$ to consist of pairs $\la s,f \ra$ where:
\begin{enumerate}
\item $s \in \mathbb P_{\kappa,\calS,\calT}$.
\item\label{dec*} $f$ is a function with $\dom(f) \in [s \times \kappa]^{<\kappa}$.  If $\la M,\alpha \ra \in \dom(f)$, then $f(M,\alpha) \in M^* \cap W$, where $M^*$ is the successor of $M$ in $s$ if it exists, $M^* = K$ if $M$ is the largest model in $s$, and $W$ is the smallest member of $\calT$ such that $M \in W$.
\item\label{decrest} $f \in K$, and for each $M \in s$, $f \rest M \in M$.
\end{enumerate}
In order for (\ref{dec*}) to make sense, we assume that $\calT$ is linearly ordered by $\in$ and $\subseteq$-cofinal in $K$.
We note that by Claim 2.34 of \cite{neeman}, if $\la s,f \ra$ is a condition and $M \in s$, then $\ran(f \rest M) \subseteq M$, so (\ref{decrest}) is equivalent to saying that $f \cap M \in M$ for all $M \in s$.
We put $\la t,g \ra \leq \la s,f \ra$ when $s \subseteq t$ and $f \subseteq g$.
For a model $M \in \calS\cup\calT$ and a condition $p = \la s,f\ra$, we sometimes write $p \rest M$ for $\la s \cap M, f \cap M \ra$.

Essentially, the only differences between our decorated poset and that of \cite{neeman} are that we consider a smaller class of decorations, namely those ${<}\kappa$-sized sets that are partial functions on $\kappa$, and we don't allow the decorations attached to a model $M$ to jump past the next transitive model above $M$.  One can check that the proof of strong properness for $\calS \cup \calT$, in particular Claim 2.38 of \cite{neeman}, still holds for this modified poset.  (Almost no change to the argument is needed to accommodate our modified decorations.)  Let us state this result:
\begin{lemma}
\label{spdec}
Suppose $\la s,f \ra \in \mathbb P^\mathrm{dec*}_{\kappa,\calS,\calT}$ and $M \in s$.  Then $\la s \cap M, f \cap M \ra$ is a condition, and if $\la t,g \ra \in M$ and $\la t,g \ra \leq \la s\cap M, f \cap M \ra$, then $\la t,g \ra$ is compatible with $\la s,f \ra$.  Furthermore, these conditions have a greatest lower bound $\la r,h \ra$, where $r$ is the closure of $s \cup t$ under intersections, $r \cap M = t$, and $h = f \cup g$.
\end{lemma}

When we take a generic filter $G$ for $\mathbb P_{\kappa,\calS,\calT}$ or $\mathbb P^\mathrm{dec*}_{\kappa,\calS,\calT}$, we will say that a model $M$ ``appears in $G$''  to mean, if we are using the undecorated poset, that there is some $s \in G$ with $M \in s$, and if we are using the decorated poset, that there is some $\la s,f \ra \in G$ with $M \in s$.

As noted in \cite{neeman}, if each $M \in \calS$ is ${<}\kappa$-closed then the whole forcing is ${<}\kappa$-closed.  
%Further, the requirements in this case are redundant; for instance the last requirement on the decorations is implied by Claim 2.34 of \cite{neeman}.
The key points of these modifications are as follows:
\begin{enumerate}
\item The union of the decorations appearing in a generic that are attached to a model $M$ will be a surjection from $\kappa$ to the successor model $M^*$, which in typical situations must be a small model.
\item If $K = H_\theta$, $\theta^{<\theta} = \theta$, and $\calT$ is sufficiently rich, then the poset satisfies the $\theta$-c.c.  This does not hold for the version of the decorated poset appearing in \cite{neeman}, since, for $\kappa = \omega$, it adds a club subset of $\theta$ that contains no infinite ground model set.
\item %We lose some continuity of the generic model sequence compared to the version appearing in \cite{neeman}.  
%For $\kappa = \omega$, the collection of $\sup(M \cap \ord)$ for $M$ appearing in the generic will not contain all its limit points of countable cofinality.  However,
In typical situations, the sequence of small models appearing between two consecutive transitive models will be continuous at limits of cofinality $\kappa$.
\end{enumerate}

Suppose $K = H_\theta$ for some regular $\theta$.  Then $\theta^{<\theta} = \theta$ if and only if there is a continuous $\in$-increasing sequence of transitive sets $\vec W = \la W_\alpha : \alpha < \theta \ra$ such that each $W_\alpha \in H_\theta$ and $\bigcup_{\alpha<\theta} W_\alpha = H_\theta$.  Let us call such a sequence a \emph{filtration} of $H_\theta$.  %Note that any two filtrations are equal on a club.  
In case $\theta$ is an inaccessible cardinal, then we can simply take $W_\alpha=V_\alpha$.  Otherwise, it will be useful to take $\vec W$ as a predicate and consider collections of elementary submodels of $\la H_\theta,\in,\vec W \ra$.  When $\alpha^{<\kappa} <\theta$ for each $\alpha<\theta$, each $W_\alpha \prec \la H_\theta,\in,\vec W \ra$ with $\cf(\alpha) \geq \kappa$ is ${<}\kappa$-closed.  This is because for any $x \in [W_\alpha]^{<\kappa}$, there is some $\beta<\alpha$ such that $x \subseteq W_\beta$, and $[W_\beta]^{<\kappa} \subseteq W_\alpha$ by elementarity.  Thus the set of such $W_\alpha$'s can be part of an appropriate pair for $\kappa$ and $H_\theta$.

\begin{lemma}
\label{rankmeet}
Suppose $\theta$ is regular, $\vec W = \la W_\alpha : \alpha < \theta \ra$ is a filtration of $H_\theta$, $\frak A = \la H_\theta,\in,\vec W \ra$, $M \prec \frak A$, and $\alpha \in M$.  Then $M \cap W_\alpha \prec \frak A$ if and only if $W_\alpha \prec \frak A$.  Furthermore, if $W_\beta \prec \frak A$ and $\alpha = \min(M \setminus \beta)$, then $W_\alpha \prec \frak A$.
\end{lemma}

\begin{proof}
We use the Tarski-Vaught criterion.  Suppose $W_\alpha \prec \frak A$.  Let $a \in M \cap W_\alpha$ and suppose $\frak A \models \exists x \varphi(a,x)$.  Then $W_\alpha \models  \exists x \varphi(a,x)$, and $M \models (\exists x \in W_\alpha) \varphi(a,x)$.  Thus there is $b \in M \cap W_\alpha$ such that $\frak A \models \varphi(a,b)$.

Now suppose $M \cap W_\alpha \prec \frak A$.  If $W_\alpha \nprec \frak A$, then there is some $a \in W_\alpha$ and a formula $\varphi(x,y)$ such that $\frak A \models \exists x \varphi(x,a)$, but a witness cannot be found in $W_\alpha$.  Then 
$$M \models (\exists y \in W_\alpha)(\exists x \varphi(x,y) \wedge (\forall z \in W_\alpha)\neg\varphi(z,y))$$
Let $b \in M \cap W_\alpha$ witness the outermost quantifier.  By our suppositions, there is $c \in M \cap W_\alpha$ such that $\frak A \models \varphi(c,b)$.  This contradicts that $\frak A \models (\forall z \in W_\alpha)\neg\varphi(z,b))$.

Now suppose $W_\beta \prec \frak A$ and $\alpha = \min(M \setminus \beta)$.  If $W_\alpha \nprec \frak A$, there is $a \in W_\alpha$ and a formula $\varphi(x,y)$ such that $\frak A \models \exists y\varphi(a,y) \wedge  (\forall y \in W_\alpha) \neg \varphi(a,y)$.  By elementarity, there is $b \in M \cap W_\alpha = M \cap W_\beta$ such that
$M \models \exists y \varphi(b,y) \wedge (\forall y \in W_\alpha) \neg \varphi(b,y)$.
Since $W_\beta \prec \frak A$, there is $c \in W_\beta \subseteq W_\alpha$ such that $\frak A \models \varphi(b,c)$.
But this contradicts that $\frak A \models (\forall y \in W_\alpha)\neg\varphi(b,y))$.
\end{proof}

Let us say that a tuple $\la \kappa,\lambda,\theta,\vec W,\calS,\calT \ra$ is \emph{nice} when:
\begin{enumerate}
\item $\kappa<\lambda<\theta$ are regular cardinals;
\item $\alpha^{<\kappa} < \theta$ for each $\alpha<\theta$;
\item $\vec W = \la W_\alpha : \alpha <\theta \ra$ is a filtration of $H_\theta$;
\item $\calS$ is a set of $M \prec \frak A = \la H_\theta,\in,\vec W \ra$ such that each $M \in \calS$ satisfies $|M| <\lambda$ and $M \cap \lambda \in \lambda$;
\item $\calT = \{ W_\alpha \prec \frak A : \cf(\alpha) \geq \lambda \}$;
\item $\la \calS,\calT \ra$ is appropriate for $\kappa$ and $\frak A$.
\end{enumerate}
The following is an extension of Claim 5.7 of \cite{neeman}:

\begin{lemma}
\label{allT}
Suppose $\la \kappa,\lambda,\theta,\vec W,\calS,\calT \ra$ is nice.   Let $\mathbb P$ be either $\mathbb P_{\kappa,\calS,\calT}$ or $\mathbb P_{\kappa,\calS,\calT}^\mathrm{dec*}$.  Then for any $p \in \mathbb P$ and $W \in \calT$, $p$ can be extended to include the model $W$.
\end{lemma}

\begin{proof}
We give the proof for the decorated poset; the other case is just slightly simpler.  
Let $\la s,f \ra \in \mathbb P$ and let $W_\alpha \in \calT$.  First suppose $s \subseteq W_\alpha$.  Since for each $\la M,\alpha\ra \in \dom(f)$, $f(M,\alpha)$ is required to be in the smallest $W \in \calT$ such that $M \in W$, $f \subseteq W_\alpha$.  Since $W_\alpha$ is ${<}\kappa$-closed, $\la s,f \ra \in W_\alpha$.  Thus $\la s \cup \{W_\alpha\}, f \ra$ is a condition witnessing the claim.  If the claim fails for $W_\alpha$, then it must be witnessed by $\la s,f \ra$ with $s \nsubseteq W_\alpha$.  
Let us assume that $\la s,f \ra$ is a witness to failure with $\min \{ \rank(M) : M \in s \setminus W_\alpha \}$ as small as possible.  
If $M$ is of minimal rank $\geq\alpha$ in $s$, then $\rank(M) >\alpha$.  If $\alpha \in M$, then 
$\la s \cap M \cup \{W_\alpha \}, f \cap M \ra$
%f \rest M \cup \{ \la W_\alpha,\emptyset \ra \} \ra$$
is a condition in $M$ below $\la s \cap M, f \cap M \ra$, since $f \cap M \subseteq W_\alpha$.  Thus by Lemma \ref{spdec}, it is compatible with $\la s,f \ra$.

Therefore, we must have $\alpha \notin M$.  Thus $M$ must be in $\calS$.  Let $\beta = \min(M \cap \theta \setminus \alpha)$.  Since $M$ is not cofinal in $\beta$, $\cf(\beta) \geq \lambda$.  
By Lemma \ref{rankmeet}, $W_\beta \in \calT$.  Thus $\la s \cap M, f \cap M \ra$ can be extended in $M$ to include $W_\beta$.
By Lemma \ref{spdec}, there is a condition $p \leq \la s,f \ra$ that includes $W_\beta$.  By the minimality assumption on $\la s,f \ra$, there is $q \leq p$ that includes $W_\alpha$.  This contradicts that $\la s,f \ra$ witnesses the failure of the claim.
\end{proof}

\begin{corollary}
\label{thetacc}
Under the hypotheses of Lemma \ref{allT}, $\mathbb P$ is $\theta$-c.c.
\end{corollary}

\begin{proof}
Suppose that $\calA \subseteq \mathbb P$ is a maximal antichain.  Let $W \in \calT$ be such that $\calA \cap W$ is maximal in $\mathbb P \cap W$.  Let $p \in \mathbb P$ be arbitrary, and extend it to $q$ that includes $W$.  There is $a \in \calA \cap W$ that is compatible with $q \rest W$.   By Lemma \ref{spdec}, $a$ is compatible with $q$.  Thus $\calA \subseteq W$.
\end{proof}

\begin{lemma}
\label{contchains}
Suppose $\la \kappa,\lambda,\theta,\vec W,\calS,\calT  \ra$ is nice and $\calS$ is stationary.
%Assume that $\mathbb P^\mathrm{dec*}_{\kappa,\calS,\calT}$ preserves the regularity of $\kappa$.
Let $G \subseteq \mathbb P^\mathrm{dec*}_{\kappa,\calS,\calT}$ be generic, let $W_0 \in W_1$ be two consecutive transitive models of $\calT$, and let $(W_0,W_1)_G$ be the set of models appearing in $G$ between $W_0$ and $W_1$.  Then $(W_0,W_1)_G$ is a subset of $\calS$ that is:
\begin{enumerate}
\item linearly ordered by both $\in$ and $\subseteq$, 
\item $\subseteq$-cofinal in $W_1$, and
%\item of order-type $\lambda$, and
\item continuous at limit points of $V$-cofinality at least $\kappa$.
\end{enumerate}
\end{lemma}

\begin{proof}
The linearity of $(W_0,W_1)_G$ follows from \cite[Claim 2.10]{neeman}.  To show that it is cofinal in $W_1$, let $\la s,f \ra$ be a condition such that $W_0,W_1 \in s$ and let $x \in W_1$.  By the stationarity of $\calS$, there is $M \in \calS$ such that $\{ \la s,f \ra,x \} \in M$.  Applying Lemma \ref{spdec}, there is $\la t,g\ra \leq \la s,f \ra$ with $M \in t$.  We must have $x \in M \cap W_1 \in t$.  Thus genercity implies that the union of $(W_0,W_1)_G$ is $W_1$.
%To show that the order-type of $(W_0,W_1)_G$

Now suppose that $\la s,f \ra$ forces 
%that $W_0,W_1$ are consecutive transitive models appearing in $G$, and 
that $M$ is a model in the interval $(W_0,W_1)_G$ whose index in the increasing enumeration following its $\in$-ordering is a limit ordinal $\eta$ such that $\cf^V(\eta) \geq \kappa$.  Let $x \in M$ be arbitrary.  
%By Lemma \ref{allT}, $W_1$ is the smallest transitive model in $\calT$ above $W_0$.  
Since there are no transitive models between $W_0$ and $M$, $s \cap M$ includes the interval $[W_0,M)_s$.  Since $|s| <\kappa$, there is $\la t,g \ra \leq \la s,f \ra$ such that $[W_0,M)_s$ is contained in a strict initial segment of $[W_0,M)_t$.  Let $N$ be the first model of $t$ above $[W_0,M)_s$.  Let $r \in M$ be the initial segment of $t \cap M$ up to and including $N$.  Then $r \in \mathbb P_{\kappa,\calS,\calT} \cap M$.
Let $h$ be the function $g \rest (r \times \kappa) \cup \{ \la \la N,0 \ra, x \ra \}$.
Then $\la r,h \ra$ is a condition in $M$ below $\la s \cap M,f \cap M \ra$.  By Lemma \ref{spdec}, it is compatible with $\la s,f \ra$.  For any condition $\la s',f' \ra$ below both $\la r,h\ra$ and $\la s,f \ra$ in which $N$ is not the largest model below $M$, $x$ is a member of the successor of $N$ in $s'$.  
Since $\la s,f \ra$ was an arbitrary condition forcing $M$ to be at place $\eta$ in $(W_0,W_1)_G$, it follows that models appearing at such places must be the union of the models in $(W_0,W_1)_G$ appearing below them.
\end{proof}

\begin{lemma}
\label{cards}
If $\la \kappa,\lambda,\theta,\vec W,\calS,\calT \ra$ is nice and $\calS$ is stationary, then $\mathbb P^\mathrm{dec*}_{\kappa,\calS,\calT}$ forces that $\lambda = \kappa^+$ and $\theta = \lambda^+$.
\end{lemma}

\begin{proof}
Since $\mathbb P = \mathbb P^\mathrm{dec*}_{\kappa,\calS,\calT}$ is strongly proper for $\calS \cup \calT$, the regularity of $\lambda$ and $\theta$ are preserved.  Since $\calS$ is stationary, for every $p \in \mathbb P$, and every $\alpha <\lambda$, there is $M \in \calS$ such that $\alpha,p \in M$.  Let $q \leq p$ be such that $M$ appears in $q$, $M$ has a successor $N$ in $q$, and $q$ forces that $N$ is the next model above $M$ appearing in the generic.  Then a density argument shows that $q$ forces that the union of the decorations attached to $M$ will be a surjection from $\kappa$ to $N$.  Thus $\mathbb P$ forces that $|\alpha| \leq \kappa$.  
Now let $\alpha < \theta$ and $p \in \mathbb P$ be arbitrary.  Let $W \in \calT$ be such that $\alpha,p \in W$.  Let $W'$ be the next transitive model above $W$.  Then by Lemma \ref{contchains}, $\mathbb P$ forces that $W'$ is the union of a $\subseteq$-increasing chain of sets of size $<\lambda$, so it forces that $|\alpha| \leq \lambda$.
\end{proof}

The remainder of this section borrows ideas from \cite{boban}.
Suppose $\la \kappa,\lambda,\theta,\vec W,\calS,\calT \ra$ is nice, and let $\mathbb P$ be either $\mathbb P_{\kappa,\calS,\calT}$ or $\mathbb P_{\kappa,\calS,\calT}^\mathrm{dec*}$.  Suppose $G \subseteq \mathbb P$ is generic over $V$. 
For notational convenience, let $W_0 = \emptyset$ and $W_\theta = H_\theta$.
Let us say $\delta \leq \theta$ is \emph{relevant} if either $\delta = 0$, $\delta = \theta$, or $W_\delta \in \calT$.
We define a decreasing sequence of sets $\calS_\delta^G$ for relevant $\delta \leq \theta$:  
%Let $\calS_0^G = \calS$.  
%$If $W_\delta \in \calT$ or $\delta = \theta$,
 $\calS_\delta^G = \{ M \in \calS : M \cap W_\delta$ appears in $G \cap W_\delta \}$.  
%Let $\calS_\theta^G = \{ M \in \calS : M$ appears in $G \}$.  
Note that for $W_\delta \in \calT$, $G \cap W_\delta$ is $(\mathbb P \cap W_\delta)$-generic over $V$ by strong properness, and $\calS_\delta^G \in V[G\cap W_\delta]$.

\begin{lemma}
\label{statpres}
Suppose $G \subseteq \mathbb P$ is generic.  If $\delta$ is relevant and $\calS' \subseteq \calS_\delta^G$ is stationary in $V[G \cap W_\delta]$, then $\calS' \cap \calS^G_\theta$ is stationary in $V[G]$.
\end{lemma}

\begin{proof}
Suppose $G$, $\delta$, and $\calS'$ are as hypothesized.  Let $F : W_\theta^{<\omega} \to W_\theta$ be a function in $V[G]$.  Work in $V[G \cap W_\delta]$.   
%Let $\mathbb P / (G \cap W_\delta)$ be the quotient poset of conditions $p$ in which $V_\delta$ appears and $p \restriction V_\delta \in G \cap V_\delta$.   
Let $\dot F$ be a $(\mathbb P / (G \cap W_\delta))$-name for $F$, and let $p_0 \in \mathbb P / (G \cap W_\delta)$ be arbitrary.  Let $\theta^*>\theta$ be regular and let $\frak A$ be the structure $\la H_{\theta^*},\in,\mathbb P,G\cap W_\delta,p_0,\dot F \ra$ (as defined in $V[G \cap W_\delta]$).  Let $N \prec \frak A$ be such that $N \cap W_\theta = M \in \calS'$.  
Now go back to $V$.  Let $s$ be the set of models of $p_0$.  By Lemma \ref{spdec} (or \cite[Corollary 2.32]{neeman}), there is $p_1 \leq p_0$ in $\mathbb P$ in which $M$ appears, and if $t$ is the set of models of $p_1$, then $t$ is the closure of $s \cup \{M \}$ under intersections, and if we are using the decorated poset, then the decorating function is unchanged.
%attached to models in $s$ are unchanged and are empty on the models in $t \setminus s$.
By hypothesis, $M \cap W_\delta$ appears in $G \cap W_\delta$, so $p_1 \rest W_\delta \in G \cap W_\delta$, since $p_1 \rest W_\delta$ is the weakest condition extending $p_0 \rest W_\delta$ in which $M \cap W_\delta$ appears.
$p_1$ is a strong master condition for $M$, and thus it forces that $M$ is closed under $\dot F$.  By the arbitrariness of $p_0$, it is forced by $\mathbb P / (G \cap W_\delta)$ that there is a model in $\calS' \cap \calS_\theta^{\dot G}$ that is closed under $\dot F$.
\end{proof}

\begin{lemma}
\label{*proj}
Suppose $\la \kappa,\lambda,\theta,\vec W,\calS,\calT \ra$ is nice and $\calS$ is stationary. 
 Assume also that 
 $\mathbb P_{\kappa,\calS,\calT}^\mathrm{dec*}$ preserves the regularity of $\kappa$.  If $G \subseteq \mathbb P_{\kappa,\calS,\calT}^\mathrm{dec*}$ is generic over $V$, then in $V[G]$, $\ns_\lambda \restriction \cof(\kappa)$ is the canonical projection of $\ns \restriction \calS^G_\theta$.  Equivalently, for every stationary $A \subseteq \lambda \cap \cof(\kappa)$ in $V[G]$, $\{ M \in \calS^G_\theta : M \cap \lambda \in A \}$ is stationary.
\end{lemma}

\begin{proof}
Suppose $A \subseteq \lambda \cap \cof(\kappa)$ is a stationary set in $V[G]$.  Let $\delta$ be such that $W_\delta \in \calT$ and $A \in V[G \cap W_\delta]$.  Let $\calS^G_\delta(A) = \{ M \in \calS^G_\delta : M \cap \lambda \in A \}$.  By Lemma \ref{statpres}, it suffices to show that $\calS^G_\delta(A)$ is stationary in $V[G \cap W_\delta]$.
Note that for all $\eta$ such that $\delta<\eta<\theta$ and $W_\eta \in \calT$, $\{ M \cap W_\eta : M \in \calS^G_\delta(A)  \} = \calS^G_\delta(A) \cap W_\eta$.  If $\delta<\eta<\eta'<\theta$ and $\calS^G_\delta(A) \cap W_{\eta'}$ is a stationary subset of $\p_\lambda(W_{\eta'})$, then $\calS^G_\delta(A) \cap W_{\eta}$ is a stationary subset of $\p_\lambda(W_{\eta})$.  

Suppose towards a contradiction that $\calS^G_\delta(A)$ is nonstationary in $V[G \cap W_\delta]$, and let $F : W_\theta^{<\omega} \to W_\theta$ be a function in $V[G \cap W_\delta]$ such that no $M \in \calS^G_\delta(A)$ is closed under $F$.  
There is $\eta$ such that $\delta<\eta<\theta$ and $W_\eta$ is closed under $F$, and thus $F \rest W_\eta$ witnesses that $\calS^G_\delta(A) \cap W_{\eta}$ is nonstationary in $\p(W_\eta)$.
 %no $M \in \calS^G_\delta(A) \cap W_{\eta}$ can be closed under $F$.  
 Hence $\calS^G_\delta(A) \cap W_{\eta'}$ is nonstationary for $\eta\leq\eta'<\theta$.  Now let $W_{\eta'}$ be the successor of $W_\eta$ in $\calT$.  By Lemma \ref{contchains}, the models in $(W_\eta,W_{\eta'})_G$ form a chain of length $\lambda$ that is $\in$- and $\subseteq$-increasing and continuous at points of cofinality $\kappa$.  In $V[G]$, let $\{ M_\alpha : \alpha < \lambda \}$ enumerate this chain.  Since $A$ is stationary in $V[G]$, $\{ \alpha \in A : M_\alpha \cap \lambda = \alpha \}$ is also stationary in $V[G]$.  But this means that $\calS^G_\delta(A) \cap W_{\eta'}$ is stationary in $V[G]$, a contradiction.
\end{proof}

%We will also use the following observation due to Veli\v{c}kovi\'c \cite[Claim 2.7]{boban}:
%
%\begin{lemma}
%\label{bobanclaim}
%Suppose $\calS,\calT$ are appropriate for $\kappa$ and $K$ and $\lambda$ is a cardinal such that each model $M \in \calS$ has $M \cap \lambda \in \lambda$.  If $p$ is a condition in $\mathbb P_{\kappa,\calS,\calT}$, $M$ and $N$ are in $p \cap \calS$, and $M \cap \lambda < N \cap \lambda$, then $M \cap N \in N$.
%\end{lemma}

\section{End-extending cardinals and $\ns_{\omega_1}$}
\label{sec:ee}
In this section, we investigate a relatively weak large cardinal notion and its forcing applications for the nonstationary ideal on $\omega_1$.

%For the purposes of warm-up and motivation, we investigate in this section a relatively weak large cardinal notion and its forcing applications for the nonstationary ideal on $\omega_1$.

For two sets $M,N$ we say that $N$ is an end-extension of $M$, or $M \sqsubseteq N$, when $N \cap \sup(M \cap \ord) = M \cap \ord$.
Suppose $\kappa$ is a cardinal and $\calS \subseteq [\kappa]^{<\kappa}$.  We say that $\calS$ \emph{admits gap end-extensions} when for every structure $\frak A$ on $\kappa$ in a countable language, there is an expansion $\frak A^*$ of $\frak A$, also in a countable language, such that for every $M \in \calS$ elementary in $\frak A^*$, there are cofinally many $\alpha<\kappa$ such that for some $N \prec \frak A^*$, $M \cup \{\alpha\} \subseteq N$ and $N \cap \alpha = M$.
For a cardinal $\mu\leq\kappa$, we say that $\kappa$ is $\mu$-end-extending when $[\kappa]^{<\mu}$ admits gap end-extensions, and we say $\kappa$ is end-extending when it is $\kappa$-end-extending.

\begin{lemma}
\label{statgap}
Suppose $\mu<\kappa$ and $\calS \subseteq [\kappa]^{<\kappa}$ admits gap end-extensions.  Then for all structures $\frak A$ on $\kappa$ and all clubs $C \subseteq \kappa$, there is an expansion $\frak A^*$ such that for all $M \prec \frak A^*$ with $M \in \calS$ and $M \cap \mu \in \mu$, there are cofinally many $\alpha \in C$ such that $\cf(\alpha)\geq\mu$ and $\sk^{\frak A^*}(M \cup \{\alpha\}) \cap \alpha = M$.
\end{lemma}

\begin{proof}
Let $f : \kappa^2 \to \kappa$ be such that if $\cf(\alpha) \leq \beta$, then $\{ f(\alpha,\gamma) : \gamma <\beta \}$ is cofinal in $\alpha$.
Let $\frak A$ be any structure on $\kappa$ with $f$, $C$, and $\mu$ in its language.  Then for $\alpha<\kappa$, $\cf(\alpha)$ is definable in $\frak A$ as the least $\beta$ such that for all $\gamma<\alpha$, there is $\delta<\beta$ with $f(\alpha,\delta) \geq \gamma$.  Let $\frak A^*$ be an expansion witnessing that $\calS$ admits gap end-extensions.  Suppose $M \prec \frak A^*$ is in $\calS$ and $M \cap \mu \in \mu$.  For $\alpha<\kappa$, let $N_\alpha = \sk^{\frak A^*}(M \cup \{\alpha \})$.  If $\alpha > \sup(M)+1$ is such that $N_\alpha \cap \alpha = M$, then $\cf(\alpha) \geq \mu$, because otherwise, $\cf(\alpha) \in N_\alpha \cap \mu$ and $\{ f(\alpha,\delta) : \delta < \cf(\alpha) \}$ is cofinal in $\alpha$ and contained in $N_\alpha$.  Also, if $N_\alpha \cap \alpha = M$, then $N_\alpha \models$ ``$C$ is unbounded in $\alpha$,'' so $\alpha \in C$ by elementarity.
\end{proof}

\begin{prop}
If $\mu$ is regular, $\mu<\kappa$, and $\kappa$ is $\mu$-end-extending, then $\cf(\kappa)>\mu$.
If $\kappa$ is end-extending, then $\kappa$ is weakly inaccessible.
\end{prop}

\begin{proof}
Suppose $\mu$ is regular, $\mu<\kappa$, and $\kappa$ is $\mu$-end-extending.  Towards a contradiction, suppose that $\cf(\kappa) \leq \mu$.  Let $C \subseteq \kappa$ be a club consisting of ordinals of cofinality $<\mu$. 
By the previous lemma, there is a set $M \in \p_\mu(\kappa)$ and an $\alpha \in C$ above $\sup(M)$ such that $\cf(\alpha) \geq \mu$.  This is a contradiction.

Suppose $\kappa$ is end-extending.  Then it is $\mu$-end-extending for all $\mu<\kappa$, so by the previous paragraph, $\kappa$ is regular.  Note that $\kappa>\omega_1$, since we can take $\frak A$ on $\omega_1$ such that all substructures of $\frak A$ are transitive, while admitting gap end-extensions requires non-transitive substructures.  
If $\kappa = \nu^+$, let $f(x,y)$ be a function such that for all $\alpha<\kappa$, $f(\alpha,\cdot)$ is an injection of $\alpha$ into $\nu$.  Let $\frak A$ incorporate $f$ in its language, and let $\frak A^*$ be given by the hypothesis.  Let $M \prec \frak A^*$ be such that $\nu \in M$ and $| M \cap \nu | < \nu$.  We can inductively build $N \prec \frak A^*$ such that $N$ is cofinal in $\kappa$ and $M \sqsubseteq N$, by applying end-extendibility at successor stages and taking unions at limits.  Let $\alpha\in N$ be such that $|N \cap \alpha| \geq \nu$.  Then $f(\alpha,\cdot)$ injects $N \cap \alpha$ into $N \cap \nu$, which is a contradiction since $N \cap \nu = M \cap \nu \in [\nu]^{<\nu}$.
\end{proof}

\begin{prop}
Suppose $\kappa$ is $\mu$-end-extending.
\begin{enumerate}
%\item $\omega_1 \leq \mu$ and $\omega_1 < \kappa$.
\item For all infinite cardinals $\nu<\delta<\kappa$, $(\kappa,\delta) \chang (\mu,\nu)$.
\item $0^\sharp$ exists.
\item If $\mu = \kappa$, then $\kappa$ is Rowbottom and weakly $\kappa$-Mahlo.
\end{enumerate}
\end{prop}

\begin{proof}
%For the first item, we can take an algebra $F$ on $\kappa$ such that every $z \subseteq \kappa$ closed under $F$ is infinite (implying $\mu \geq \omega_1$) and has transitive intersection with $\omega_1$. If $[\kappa]^{<\mu}$ admits gap end-extensions, then there are nontransitive $z \subseteq \kappa$ that are closed under $F$, so $\kappa > \omega_1$.

For the first item, let $\frak A$ be any structure on $\kappa$ and let $\frak A^*$ be an expansion witnessing that $\kappa$ is $\mu$-end-extending.  Let $M \prec \frak A^*$ be such that $\delta \in M$ and $|M \cap \delta| = \nu$.  Build a strictly $\sqsubseteq$-increasing continuous sequence $\la M_i : i \leq \mu \ra$ with $M_0 = M$.  Then $M_\mu \prec \frak A$, $|M_\mu \cap \kappa| = \mu$, and $|M_\mu \cap \delta| = \nu$.

The second item follows from $(\kappa,\delta) \chang (\mu,\nu)$, using Theorem 18.27 and the argument for Corollary 18.29 in \cite{jech}.

For the third item, let $c$ be any coloring of the finite subsets of $\kappa$ in $\lambda$-many colors, $\lambda<\kappa$, and let $\frak A$ be a structure incorporating $c$ into its language, such that every submodel of size $<\kappa$ can be end-extended.  By repeatedly end-extending a countable $M \prec \frak A$ with $\lambda \in M$, we obtain a subset of $\kappa$ of size $\kappa$ on which $c$ takes only countably many values.  This implies that $\kappa$ is a regular Rowbottom cardinal, and by Shelah \cite{shelah}, that it is weakly $\kappa$-Mahlo.
\end{proof}

The following is inspired by arguments in \cite{cox} and \cite{sakai}.
 
\begin{prop}
Suppose $\calI$ is a $\kappa$-complete ideal on $\kappa$, $\calS^* \subseteq [H_{(2^{2^\kappa})^+}]^{<\kappa}$ is stationary, $\p(\kappa)/\calI$ is $\calS^*$-proper, and $\calS = \{ M \cap\kappa : M \in \calS^* \}$.  Then $\calS$ admits gap end-extensions.
\end{prop}

\begin{proof}
First we claim that $\calI$ is precipitous.  Let $X$ be an $\calI$-positive set, and suppose $X \Vdash$ ``$\vec\tau = \la \tau_n : n \in \omega \ra$ is a descending sequence of ordinals in the generic ultrapower.''  Let $\calD_n$ be the dense set of conditions deciding $\tau_n = \check f$ for some function $f$ on $\kappa$.  
Let $\frak A = \la H_{(2^{2^\kappa})^+},\in,\lhd,\calI,X,\vec\tau \ra$, where $\lhd$ is a well-order, and let $M \prec \frak A$ be in $\calS^*$.  Let $Y \subseteq X$ be $(M,\p(\kappa)/\calI)$-generic.

Note that for each dense set $\calD \subseteq \p(\kappa)/\calI$ in $M$, $Y \setminus \bigcup (M \cap \calD) \in \calI$.  Otherwise, there is an $\calI$-positive $Z \subseteq Y$ such that $Z \cap A = \emptyset$ for each $A \in \calD \cap M$.  But this contradicts that $Z$ is $(M,\p(\kappa)/\calI)$-generic.
By $\kappa$-completeness, there is $\alpha \in Y$ such that for each dense $\calD \in M$, there is $A \in \calD \cap M$ with $\alpha \in A$.  Let $\calU = \{ A \in \p(\kappa) \cap N : \alpha \in A \}$.  Then $\calU$ is an $(M,\p(\kappa)/I)$-generic ultrafilter containing $X$.  For each $n$, let $f_n$ be a function such that some $A \in \calD_n \cap \calU$ decides $\tau_n = \check f_n$.  By the statement forced by $X$, $\{ \beta : f_{n+1}(\beta) \in f_n(\beta) \} \in \calU$ for each $n$.  So $\la f_n(\alpha) : n \in \omega \ra$ is a descending sequence of ordinals, a contradiction.

Now let $\frak A$ be a structure on $\kappa$ let $\xi<\kappa$.  Let $\frak B = \la H_{(2^{2^\kappa})^+},\in,\lhd,\calI,\frak A \ra$, where $\lhd$ is a well-order.  Let $\frak A^*$ be the result of restricting the Skolem functions of $\frak B$ to $\kappa$.
%so that for any $N \prec \frak A^*$, $\sk^{\frak B}(N) \cap \kappa = N$.  
Let $N^* \in \calS^*$ be elementary in $\frak B$, and let $G$ be a $\p(\kappa)/\calI$-generic filter with a master condition for $N^*$.  Let $N=N^*\cap\kappa\prec\frak A^*$. 

Let $j : V \to M$ be the generic ultrapower embedding.  Since $|N^*|<\kappa$, $j(N^*) = j[N^*]$.  Suppose $\alpha$ is an ordinal in $Q = \sk^{j(\frak B)}(j(N^*) \cup \{\kappa \})$.  Then there is a function $f : \kappa \to \ord$ in $N^*$ such that $\alpha = j(f)(\kappa)$.  Let $\tau \in N^*$ be a $\p(\kappa)/\calI$-name for $j_{\dot{G}}(f)(\kappa)$.  Since $G$ has a master condition for $N^*$, $\tau^G \in N^*$.  Since $j(N^*) \cap \kappa = N^* \cap \kappa = N$, we have $Q \cap \kappa = N$.  By elementarity, there is $\beta$ such that $\xi<\beta<\kappa$ and $\sk^{\frak A^*}(N \cup \{\beta\}) \cap \beta = N$.
\end{proof}

It follows that measurable cardinals are end-extending.  But since $\kappa$ being end-extending depends only on $\p(\kappa)$, reflection shows that every normal ultrafilter on $\kappa$ has a measure-one set of end-extending cardinals.
  Furthermore, since we can consistently have weakly inaccessible cardinals carrying $\omega_1$-saturated ideals, end-extending cardinals are not necessarily strong limit.  Even successor cardinals can be end-extending in limited degrees.  For example, the statement that $\omega_2$ is $\omega_1$-end-extending is easily seen to be equivalent to the notion $\mathrm{SCC}^{\cof}_{\mathrm{gap}}$ defined in \cite{cox}.    It follows from the above proposition that if $\kappa$ carries a $\kappa$-complete ideal $\calI$ such that $\p(\kappa)/\calI$ is a proper forcing, then $\kappa$ is $\omega_1$-end-extending.  If $\kappa$ is measurable and $\mu<\kappa$ is regular and uncountable, then this is forced by the L\'evy collapse $\col(\mu,{<}\kappa)$ (see \cite{foreman}).

Let us now see how this notion can be used with Neeman's forcing.  Suppose $\theta^{<\theta} = \theta$ and $\alpha^\omega<\theta$ for each $\alpha<\theta$.
Let $\vec W$ be a filtration of $H_\theta$, and let $\frak A = \la H_\theta,\in,\vec W \ra$.  Let $\calS$ be the set of countable $M \prec \frak A$, and let $\calT$ be the set of $W_\alpha \prec \frak A$ such that $\cf(\alpha)$ is uncountable.  Then according to the terminology of the previous section, $\la \omega,\omega_1,\theta,\vec W,\calS,\calT \ra$ is nice.  Let us set $\mathbb C_{\vec W} = \mathbb P^\mathrm{dec*}_{\omega,\calS,\calT}$ under these hypotheses.  
If $\theta$ is inaccessible, let $\mathbb C_\theta =\mathbb C_{\vec W}$ for $\vec W = \la V_\alpha : \alpha<\theta\ra$.

\begin{theorem}
\label{wps-ee}
If $\theta$ is an inaccessible $\omega_1$-end-extending cardinal, or a stationary limit of such cardinals, then $\mathbb C_\theta$ forces that $\ns_{\omega_1}$ is weakly presaturated.  If $\theta^{<\theta} = \theta$, $\alpha^\omega<\theta$ for all $\alpha<\theta$, and $\theta$ is $\omega_1$-end-extending, then the same conclusion is forced by $\mathbb C_{\vec W}$, where $\vec W$ is any filtration of $H_\theta$.
\end{theorem}

\begin{proof}
Let us first show the latter statement.
Let $G \subseteq \mathbb C_{\vec W} = \mathbb P^\mathrm{dec*}_{\omega,\calS,\calT}$ be generic, and let $\calS^G_\theta$ be the collection of $M \in \calS$ that appear in $G$.  By Lemma \ref{cards}, $\mathbb P$ forces $\theta$ to become $\omega_2$.
By Lemma \ref{*proj}, $V[G]$ satisfies that $\ns_{\omega_1}$ is the canonical projection of $\ns \rest \calS^G_\theta$.
By Lemma \ref{wps-equivs}, it suffices to show that for all stationary $A \subseteq \omega_1$ and $f : \omega_1 \to \omega_1$, $\{ M \in \calS^G_\theta : M \cap \omega_1 \in A \wedge f(M \cap \omega_1) <\ot(M \cap \theta) \}$ is stationary.

Let $\dot A$ be a name for a stationary subset of $\omega_1$, $\dot f$ a name for a function from $\omega_1$ to itself, and $\dot F$ a name for a function from $(H_\theta^V)^{<\omega}$ to $H_\theta^V$.  Let $p_0 \in \mathbb C_{\vec W}$ be arbitrary.  Let $\lambda > \theta$ be regular and let $\frak B = \la H_\lambda,\in,\frak A,p_0,\dot f,\dot F,\lhd \ra$, where $\lhd$ is a well-order.  Let $\frak A^*$ be the result of restricting the Skolem functions for $\frak B$ to $H_\theta$.  Let $\frak A^{**}$ be an expansion of $\frak A^*$ given by the $\omega_1$-end-extending hypothesis.

In $V[G]$, there is some $M \prec \frak A^{**}$ such that $M \in S^G_\theta$ and $M \cap \omega_1 \in A$.  Let $p_1 \leq p_0$ decide the value of such $M$, force that $M \cap \omega_1 \in \dot A$, and decide $\dot f(M \cap \omega_1) = \xi$ for some ordinal $\xi<\omega_1$.  We may assume that $p_1$ takes the form $\la s_1,g \ra$ with $M \in s_1$.

Let $\alpha$ be least such that $p_1 \in W_\alpha\in\calT$.  By end-extendibility, there is $M_1 \prec \frak A^{**}$ such that $M \sqsubseteq M_1$, and if $\alpha_0$ is the least ordinal in $M_1 \setminus M$, then $\alpha\leq\alpha_0$.
By Lemma \ref{rankmeet}, $W_{\alpha_0} \in \calT$.  Now repeat this $(\xi+1)$-many times to obtain a continuous $\subseteq$-increasing sequence of models $\la M_i : i \leq \xi+1 \ra$, all elementary in $\frak A^{**}$, with $M_0 = M$, and a corresponding sequence of ordinals $\la \alpha_i : i \leq \xi \ra$ such that $\alpha_i \in M_{i+1}$, $M_{i+1} \cap W_{\alpha_i} = M_i$.  Let $M^* = M_{\xi+1}$.  Then $M^* \cap W_{\alpha_0} = M$, and $\ot(M^* \cap \theta) > \xi$.

Let $s_2 = s_1 \cup \{ W_{\alpha_0}, M^* \}$.  Clearly, $s_2$ is an $\in$-chain.  To see that it is closed under intersections, note that $W_{\alpha_0} \cap M^* = M \in s_1$, and for any $Q \in s_1$, $Q \cap W_{\alpha_0} = Q$ and $M^* \cap Q = M^* \cap W_{\alpha_0} \cap Q = M \cap Q \in s_1$.  
 Then $p_2 = \la s_2,g \ra$ is a condition below $p_1$.

Let $N^* = \sk^{\frak B}(M^*)$.  Since $p_2$ is a master condition for $N^*$ and $\dot F \in N^*$,  $p_2$ forces that $N^*[G]\cap H_\theta^V = M^*$, so it forces $M^*$ to be closed under $\dot F$.  Also, $p_2$ forces $M^*$ to be in $\calS^G_\theta$ and $M \cap \omega_1 = M^* \cap \omega_1 \in \dot A$.  Since $p_0$ and $\dot F$ were arbitrary, it is forced that the set 
$\{ Q \in \calS^G_\theta : Q \cap \omega_1 \in A \wedge f(Q \cap \omega_1) < \ot(Q \cap \theta) \}$
is stationary.

Now suppose that $\theta$ is an inaccessible stationary limit of regular $\omega_1$-end-extending cardinals.  
Let $\dot A,\dot f,\dot F,p_0$ be as above.  
By the chain condition, we may assume that $\dot A,\dot f,p_0 \in V_\theta$, and on a club of $\alpha<\theta$, $\dot F \cap V_\alpha$ is forced to be a name for $\dot F \rest V_\alpha$.  We can find an inaccessible $\kappa<\theta$ such that $\kappa$ is $\omega_1$-end-extending, $\dot A,\dot f,p_0 \in V_\kappa \in \calT$, and $\dot F \cap V_\kappa$ is a $\mathbb C_\kappa$-name for a function from $V_\kappa^{<\omega}$ to $V_\kappa$.  By the arguments of the previous paragraphs, there is $q \leq p_0$ in $V_\kappa$ that forces some $M \in \calS \cap V_\kappa$ to appear in the generic $G \cap V_\kappa$, be closed under $\dot F$, and have the properties that  $M \cap \omega_1 \in \dot A$ and  $\dot f(M \cap \omega_1) < \ot(M \cap \kappa)$.  Since $V_\kappa \in \calT$, 
$\mathbb C_\kappa$ is a regular suborder of $\mathbb C_\theta$,
and any generic $G \subseteq \mathbb C_\theta$ possessing $q$ will yield an extension satisfying these statements.  By the arbitrariness of $p_0$ and $\dot F$, the collection of such models $M$ is forced to be stationary.
\end{proof}

We remark that above forcing does not necessarily render $\ns_{\omega_1}$ precipitous.  Claverie and Schindler \cite{clav} showed that a precipitous weakly presaturated ideal (a.k.a.\ ``strong ideal'') is equiconsistent with a Woodin cardinal.

Another strengthening of weak presaturation for $\ns_{\omega_1}$ is the statement that for every $f : \omega_1 \to \omega_1$, there is a club $C \subseteq \omega_1$ and a canonical function $\psi_\alpha$ such that $f(\beta) < \psi_\alpha(\beta)$ for all $\beta \in C$.
Deiser-Donder \cite{dd} and Larson-Shelah \cite{ls} showed that this property, which the latter call ``Bounding'', is equiconsistent with an inaccessible limit of measurable cardinals.  

Besides consistency strength considerations, there is a combinatorial reason why $\mathbb C_\theta$ does not necessarily force Bounding.
Suppose $\Diamond_{\omega_1}$ holds in $V$.  Let $\la x_\alpha : \alpha < \omega_1 \ra$ be a diamond sequence.  Define a function $f : \omega_1 \to \omega_1$ as follows.  If $x_\alpha$ codes, via the G\"odel pairing function, a countable transitive set $N$, let $f(\alpha) = \ot(N \cap \omega_1)$.  Otherwise, let $f(\alpha) = 0$.

Let $\theta > \omega_1$ be regular and let $\frak A$ be a structure on $H_\theta$ in a countable language. 
Let $\la M_i : i < \omega_1 \ra$ be a continuous, $\in$-increasing sequence of countable elementary substructures of $\frak A$.
Let $C \subseteq \omega_1$ be the club of $\alpha$ such that $M_\alpha \cap \omega_1 = \alpha$.
Let $M$ be the union of the $M_i$, and let $\bar M$ be the transitive collapse of $M$.
Let $A \subseteq \omega_1$ code the $\in$-structure of $\bar M$ via the pairing function.
There is a club $D \subseteq C$ such that for all $\alpha \in D$, $A \cap \alpha$ codes a structure isomorphic to $M_\alpha$.  There is some $\alpha \in D$ such that $x_{\alpha} = A \cap \alpha$.  
We have $f(\alpha) = f(M_\alpha \cap \omega_1) =  \ot(M_{\alpha} \cap \theta)$.
Thus there are stationary-many $z \in \p_{\omega_1}(\theta)$ such that $f(z \cap \omega_1) = \ot(z \cap \theta)$.

If $\mathbb P$ is a proper forcing, then this continues to hold in $V^{\mathbb P}$ with respect to the same $f$.  In particular, it holds in any extension by $\mathbb C_\theta$.
Let $G \subseteq \mathbb P$ be generic.  There is some $M \prec H_{\omega_2}^{V[G]}$ such that $f \in M$, and $f(M \cap \omega_1) = \ot(M \cap \omega_2)$.
If Bounding were to hold in $V[G]$, then by elementarity, there would be an ordinal $\gamma \in M$ and a club $C \in M$ such that $M \models$ ``$f(\beta) < \psi_\gamma(\beta)$ for all $\beta \in C$.''  But $\alpha = M \cap \omega_1 \in C$ and $\psi_\gamma(\alpha) = \ot(M \cap \gamma) < \ot(M \cap \omega_2) = f(\alpha)$, a contradiction.

\section{$\wcc(\omega_2,\cof(\omega_1))$ from $(+2)$-subcompactness}
\label{sec:wccomega_2}

As far as we are able to ascertain, until the time of this writing, the best known upper bounds for the consistency strength of $\wcc(\mu^+,\cof(\mu))$, for $\mu$ regular uncountable, remained what was discovered in the late 1970s, namely an almost-huge cardinal.  Kunen \cite{kunen} developed a forcing strategy that collapses a huge cardinal $\kappa$ to become the successor of a chosen regular cardinal $\mu<\kappa$, and forces Chang's Conjecture $(\kappa^+,\kappa) \chang (\mu^+,\mu)$ and the existence of a saturated ideal on $\kappa$.  Shortly thereafter, Magidor showed that the saturated ideal, and consequently $\wcc(\kappa,\cof(\mu))$, can be obtained from an almost-huge $\kappa$ (see \cite{foreman}).  In \cite{eh}, the author and Hayut showed by a reflection argument that the strength of $(\omega_3,\omega_2) \chang (\omega_2,\omega_1)$ is weaker than a huge cardinal, and that a single huge cardinal can be used to get Chang's Conjecture between many pairs of cardinals simultaneously, but the reflection did not take us down to the level of almost-huge or lower.  On the other hand, we showed that if we allow some distance between the pairs of cardinals, the strength can be shown to be much lower by a very different forcing argument.  In particular, we got a model of the generalized continuum hypothesis (GCH) plus $(\omega_4,\omega_3) \chang (\omega_2,\omega_1)$ from a model of GCH with a $(+2)$-subcompact cardinal.  We get some more information using an observation of Adolf \cite{adolf}:

\begin{lemma}[Adolf]
Suppose $2^{\kappa^+} = \kappa^{++}$ and $(\kappa^{+++},\kappa^{++}) \chang (\kappa^+,\kappa)$.  Then for every regular $\mu\leq\kappa$ the set of $M \prec H_{\kappa^{+3}}$ such that $M \cap \kappa^+ \in \kappa^+$, $\cf(M \cap \kappa^{+}) = \mu$, $\cf(M \cap \kappa^{++}) = \kappa$, and $\ot(M \cap \kappa^{+3}) = \kappa^+$, is stationary.
\end{lemma}

Let us show that the above result yields many witnesses to Chang's Conjecture that are countably closed.  Suppose GCH holds and $M \prec H_{\omega_4}$ is such that $\omega_1 \subseteq M$ and $\cf(M \cap \kappa)>\omega$ for each uncountable cardinal $\kappa \in M$.  By CH, $M$ contains all reals.  Since $\cf(M \cap \omega_2)$ is uncountable, $[\omega_2]^{<\omega_1} \cap M = \bigcup \{ [\alpha]^{<\omega_1} : \alpha \in M \cap \omega_2 \} = [M \cap \omega_2]^{<\omega_1}$.  Similarly, $[\omega_3]^{<\omega_1} \cap M = \bigcup \{ [\alpha]^{<\omega_1} \cap M : \alpha \in M \cap \omega_3 \}$, and for any $\alpha \in [\omega_2,\omega_3) \cap M$, a bijection $f : \omega_2 \to \alpha$ in $M$ yields that $[\alpha]^{<\omega_1} \cap M = \{ f[z] : z \in [M \cap \omega_2]^{<\omega_1} \} = [M \cap \alpha]^{<\omega_1}$.  Thus $[\omega_3]^{<\omega_1} \cap M = [M \cap \omega_3]^{<\omega_1}$.  Similarly, $[\omega_4]^{<\omega_1} \cap M = [M \cap \omega_4]^{<\omega_1}$.  Thus $M$ is countably closed.

Now assume GCH and $(\omega_4,\omega_3) \chang (\omega_2,\omega_1)$.  Let us introduce a Neeman forcing.  Let $\vec W$ be a filtration of $H_{\omega_4}$, and let $\frak A = \la H_{\omega_4},\in,\vec W \ra$.  Let $\calS$ be the collection of all countably closed $N \prec \frak A$ of size $\omega_1$ such that $N \sqsubseteq M$ for some countably closed $M \prec \frak A$ with $\ot(M \cap \omega_4) = \omega_2$.  By taking countably closed elementary initial segments of witnesses to Chang's Conjecture, we see that $\calS$ is stationary.  Let $\calT$ be the collection of all $W_\alpha \prec \frak A$ with $\cf(\alpha) \geq \omega_2$.  Each $W \in \calT$ is closed under $\omega_1$-sequences.  For any $M \in \calS$ and $W \in \calT$, $M \cap W$ is in $\calS$, since $W \cap M \sqsubseteq M$, and it is a countably closed elementary substructure of $\frak A$.  Thus $\la \calS,\calT \ra$ is appropriate for $\omega_1$ and $\frak A$.

Let $\mathbb P = \mathbb P^\mathrm{dec*}_{\omega_1,\calS,\calT}$.  Then $\mathbb P$ is countably closed, preserves $\omega_2$ and $\omega_4$, and collapses $\omega_3$ so that $\omega_3^{V^{\mathbb P}} = \omega_4^V$.  Let $\dot F$ be a $\mathbb P$-name for a function from $(H_{\omega_4}^V)^{<\omega}$ to $H_{\omega_4}^V$.  Let $p_0 \in \mathbb P$ be arbitrary, and let $\dot f$ be a $\mathbb P$-name for a function from $\omega_2$ to $\omega_2$.
Let $\frak B = \la H_{\omega_5},\in,\lhd,\vec W,\dot F,p_0,\dot f \ra$, where $\lhd$ is a well-order.  Let $\frak B^*$ be the result of restricting the Skolem functions for $\frak B$ to $H_{\omega_4}$.

By $(\omega_4,\omega_3) \chang (\omega_2,\omega_1)$ and the above lemma, there are stationary-many countably closed $M \prec \frak B^*$ such that $\ot(M \cap \omega_4) = \omega_2$ and $|M \cap \omega_3| = \omega_1$.  For each $\alpha < \omega_4$, let $M_\alpha \prec \frak B^*$ be such a model with $\alpha \in M_\alpha$.  For each $\alpha<\omega_4$ of cofinality at least $\omega_2$, $M_\alpha \cap \alpha$ is bounded below $\alpha$.  Using Fodor's Lemma, there is a stationary $A \subseteq \omega_4 \cap \cof(\geq\omega_2)$ and a $\gamma<\omega_4$ such that $M_\alpha \cap \alpha \subseteq \gamma$ for all $\alpha \in A$.  By GCH, there is a stationary $B \subseteq A$ and a set $N$ such that $M_\alpha \cap W_\alpha = N$ for all $\alpha \in B$.  We may assume that $W_\alpha \prec \frak B^*$ for all $\alpha \in B$, which implies $N \prec \frak B^*$.  Note that $N \in \calS$.

Let $p_1 = \la s_1,h \ra$ be a condition below $p_0$ with $N \in s_1$, such that $p_1$ decides the value of $\dot f(N \cap \omega_2)$, say as $\xi<\omega_2$.  Let $\alpha \in B$ be such that $p_1 \in W_\alpha$.  Let $N^* \sqsubseteq M_\alpha$ be such that $N^* \prec M_\alpha$, $N^*$ is countably closed, and $\ot(N^* \cap \omega_4) > \xi$.  Let $s_2 = s_1 \cup \{ W_\alpha,N^* \}$.  Then $s_2$ is closed under intersections, since $N^* \cap W_\alpha = N \in s_1$, and for any $Q \in s_1$, $W_\alpha \cap Q = Q$ and $N^* \cap Q = N^* \cap W_\alpha \cap Q = N \cap Q \in s_1$.  
Then $p_2 = \la s_2,h \ra$ is a condition below $p_1$, and it forces that $N^*$ is closed under $\dot F$ and that $\dot f(N^* \cap \omega_2)< \ot(N^* \cap \omega_2)$.  Note also that $\cf(N^* \cap \omega_2) = \omega_1$.  As $p_0,\dot f$, and $\dot F$ were arbitrary, we have:
\begin{theorem}
If ZFC+GCH is consistent with a (+2)-subcompact cardinal, then ZFC is consistent with $\wcc(\omega_2,\cof(\omega_1))$.
\end{theorem}

\section{Magidor models and ambitious cardinals}
\label{sec:ambitious}

Suppose $\kappa$ is a regular cardinal.  Following \cite{mv}, we call a set $M$ \emph{$\kappa$-Magidor} if $\kappa \in M$, $M \cap \kappa \in \kappa$, $M$ is extensional, $M$ has size ${<}\kappa$, and the transitive collapse of $M$ is equal to $V_\alpha$ for some $\alpha<\kappa$.  Magidor \cite{magidor} showed that $\kappa$ is supercompact if and only if for every $\theta > \kappa$, the set of $\kappa$-Magidor $M \in \p_\kappa(V_\theta)$ is stationary.

\begin{lemma}
\label{magclosure}
Suppose $\mu<\kappa$ is a regular cardinal, $M \prec V_\theta$ is a $\kappa$-Magidor model, $\mu \leq M \cap \kappa$, and $\cf(\sup(M \cap \theta)) \geq \mu$.  Then $M$ is ${<}\mu$-closed.
\end{lemma}

\begin{proof}
Let $i : V_\eta \to M$ be the inverse of the transitive collapse of $M$. Suppose $x \in [M]^{<\mu}$.  Since $\cf(M \cap \theta)\geq\mu$, there is $\alpha < \eta$ such that $x \subseteq i(V_\alpha)$.  Then $y = i^{-1}[x] \in V_\eta$, and $x = i(y) \in M$ since $|y|<M\cap\kappa=\crit(i)$.
\end{proof}

In this section, we identify a species of supercompactness that has a characterization in terms of Magidor models that is somewhat analogous to how countable models behave with end-extending cardinals.

\begin{definition*}
A cardinal $\kappa$ is \emph{ambitious} when for all $\lambda\geq\kappa$, there is $\delta_0 \geq \lambda$ such that for all $\delta \geq \delta_0$, there is a $\lambda$-closed transitive $M$, a $\delta$-closed transitive $N$, and elementary embeddings $j : V \to M$ and $k : M \to N$
such that:
\begin{itemize}
\item $\crit(j) = \kappa$ and $\lambda < j(\kappa)$.
\item $\crit(k)>\lambda$, $k[j(\kappa)] \subseteq \delta$, and $k(j(\kappa)) > \delta$.
\end{itemize}
\end{definition*}

Recall that $\kappa$ is \emph{almost-huge} if there is an elementary $j : V \to M$ with critical point $\kappa$ and $M$ is a transitive class closed under ${<}j(\kappa)$-sequences.
The name \emph{ambitious} is chosen because it is as if such $\kappa$ are trying hard to be almost-huge.  The first embedding $j$ sends $\kappa$ to a highly closed model, but $j(\kappa)$ overshoots the closure.  The next embedding $k$ tries to make up for this by sending $j(\kappa)$ into a model whose closure is above $k(\alpha)$ for each $\alpha$ at which $j$ may have fallen short of being an almost-huge embedding.  Now $k(j(\kappa))$ may overshoot the new closure, but it looks like progress.  Indeed, this is a local property of systems of measures that characterize almost-huge cardinals (see \cite[Theorem 24.11]{kanamori}).

\begin{prop}
The following are equivalent:
\begin{enumerate}
\item\label{amb} $\kappa$ is ambitious.

\item\label{amb-uf} (High-jump property) For all $\lambda \geq \kappa$, there is $\delta_0 \geq \lambda$ such that for all $\delta \geq \delta_0$, there is a normal ultrafilter $\calU$ on $\p_\kappa(\delta)$ such that for all $f : \p_\kappa(\lambda) \to \kappa$, $\{ z \in \p_\kappa(\delta) : f(z \cap \lambda) < \ot(z) \} \in \calU$.

\item\label{amb-uf-ptwise} (Shelah property) For all $\lambda \geq \kappa$ and all $f : \p_\kappa(\lambda) \to \kappa$, there is $\delta_0 \geq \lambda$ such that for all $\delta \geq \delta_0$, there is a normal ultrafilter $\calU$ on $\p_\kappa(\delta)$ such that $\{ z \in \p_\kappa(\delta) : f(z \cap \lambda) < \ot(z) \} \in \calU$.

\end{enumerate}
\end{prop}

\begin{proof}
(\ref{amb}) $\Rightarrow$ (\ref{amb-uf}): Let $\lambda\geq\kappa$, let $\delta_0 \geq \lambda$ be given by (\ref{amb}), let $\delta\geq\delta_0$, and let $j : V \to M$ and $k : M \to N$ %$V \underset{j}{\to} M \underset{k}{\to} N$ 
also be given by (\ref{amb}).  Let $i = k \circ j$.  Let $\calU = \{ X \subseteq \p_\kappa(\lambda) : j[\lambda] \in j(X) \}$, and let $\calW =  \{ X \subseteq \p_\kappa(\delta) : i[\delta] \in i(X) \}$.  Let $\pi : \p_\kappa(\delta) \to \p_\kappa(\lambda)$ be $z \mapsto z \cap \lambda$.  Then:
\begin{align*}
X \in \calU &\Leftrightarrow j[\lambda] \in j(X) \\
		&\Leftrightarrow k(j[\lambda]) = i[\lambda] \in i(X) \text{\hspace{5mm}(since $\crit(k) > \lambda)$} \\
		&\Leftrightarrow i(\pi)(i[\delta]) \in i(X) \\
		&\Leftrightarrow \pi^{-1}[X] \in \calW
\end{align*}
%Thus $\calU$ is the canonical projection of $\calW$.  
Let $\bar k : \Ult(V,\calU) \to \Ult(V,\calW)$ be the map $[f]_\calU \mapsto [f \circ \pi]_\calW$.
Let $\ell_\calU : \Ult(V,\calU) \to M$ be $[f]_{\calU} \mapsto j(f)(j[\lambda])$, and let $\ell_\calW : \Ult(V,\calW) \to N$ be $[f]_{\calW} \mapsto i(f)(i[\delta])$.  
With $j_\calU,j_\calW$ denoting the usual ultrapower embeddings, we have the following commutative diagram:
\begin{diagram}
V			&			&\rTo^{j}				&&M					&\rTo^{k}		&N			\\
			&\rdTo^{j_\calU} \rdTo(4,2)^{j_\calW}		&&\ruTo^{\ell_\calU}	&			&\ruTo^{\ell_\calW}	\\
			&				&\Ult(V,\calU)		&\rTo^{\bar k}	&\Ult(V,\calW)			
\end{diagram}
Since $j_\calW[\delta]$ is represented by $[\id]_\calW$, the claim that for all $f : \p_\kappa(\lambda) \to \kappa$, $\{ z \in \p_\kappa(\delta) : f(z \cap \lambda) < \ot(z) \} \in \calW$ is equivalent to the claim that $\sup\bar k[j_\calU(\kappa)] \leq \delta$.
To show this, note:
\begin{align*}
\delta	&\geq		\sup\{ k(\alpha) : \alpha < j(\kappa) \} \\
		&\geq	\sup\{ k(\ell_\calU(\alpha)) : \alpha < j_\calU(\kappa) \} \\
		&=		\sup\{ \ell_\calW(\bar k(\alpha)) : \alpha < j_\calU(\kappa) \} \\
		&\geq	\sup\{ \bar k(\alpha) : \alpha < j_\calU(\kappa) \}
\end{align*}

(\ref{amb-uf}) $\Rightarrow$ (\ref{amb}): Let $\lambda \geq \kappa$, let $\delta_0\geq\lambda$ be given by (\ref{amb-uf}), let $\delta\geq\delta_0$ and let $\calW$ be a normal ultrafilter on $\p_\kappa(\delta)$ given by (\ref{amb-uf}).  Let $\pi : \p_\kappa(\delta) \to \p_\kappa(\lambda)$ be $z \mapsto z \cap \lambda$, and let $\calU$ be the projection of $\calW$ to $\p_\kappa(\lambda)$ via $\pi$.  Let $j : V \to M = \Ult(V,\calU)$ and $i : V \to N = \Ult(V,\calW)$ be the ultrapower embeddings, and let $k : M \to N$ be $[f]_\calU \mapsto [f \circ \pi]_\calW$.  Ordinals $\alpha\leq\lambda$ are represented in $\Ult(V,\calU)$ by the function $f_\alpha : z \mapsto \ot(z \cap \alpha)$, and $k(\alpha)$ is the $\calW$-equivalence class of the function $z \mapsto \ot(z \cap \lambda \cap \alpha)$, which represents $\alpha$ in $\Ult(V,\calW)$.  Thus $\crit(k) > \lambda$.  
For all $\alpha < j(\kappa)$, there is a function $f : \p_\kappa(\lambda) \to \kappa$ representing $\alpha$.  Since $\{ z \in \p_\kappa(\delta) : f\circ\pi(z) < \ot(z) \} \in \calW$, $k(\alpha)<\delta$.

The implication (\ref{amb-uf}) $\Rightarrow$ (\ref{amb-uf-ptwise}) is trivial.  For the reverse direction, let $\lambda \geq \kappa$.  Let $\lambda' = 2^{\lambda^{<\kappa}}$, and let $\delta_0 \geq \lambda'$ witness (\ref{amb-uf-ptwise}) for $\lambda'$.  Enumerate all functions from $\p_\kappa(\lambda)$ to $\kappa$ as $\{ f_\alpha : \alpha < \lambda' \}$.  Let $g : \p_\kappa(\lambda') \to \kappa$ be $z \mapsto \sup\{ f_\alpha(z \cap \lambda) : \alpha \in z \}$.  Let $\delta \geq \delta_0$ and let $\calU$ be a normal ultrafilter on $\p_\kappa(\delta)$ such that $\{ z \in \p_\kappa(\delta) : g(z \cap \lambda') < \ot(z) \} \in \calU$.  By the fineness of $\calU$, for all $\alpha<\lambda'$, $\{ z \in \p_\kappa(\delta) : f_\alpha(z \cap \lambda) < \ot(z) \} \in \calU$.
\end{proof}

\begin{theorem}
\label{amb-ee}
Suppose $\kappa<\theta$ are inaccessible.  The following are equivalent:
\begin{enumerate}
\item\label{local-amb} $V_\theta \models \kappa$ is ambitious.
\item\label{mag-ee} For all structures $\frak A$ on $V_\theta$ in a countable language, there is a $\kappa$-Magidor $M \prec \frak A$ such that for all $\alpha<\kappa$, there is a $\kappa$-Magidor $N \prec \frak A$ such that $M \sqsubseteq N$ and $\ot(N \cap \theta) > \alpha$.
\end{enumerate}
Furthermore, in (\ref{mag-ee}), we can take $M$ and $N$ to be ${<}(M\cap\kappa)$-closed.  
%If $\theta$ is Mahlo, we can moreover require that $\ot(M \cap \theta)$ and $\ot(N \cap \theta)$ are inaccessible.
\end{theorem}

\begin{proof}
Suppose (\ref{local-amb}).  If (\ref{mag-ee}) fails, then there is a structure $\frak A$ on $V_\theta$ in a countable language extending $(V_\theta,\in,\kappa)$ such that for every $\kappa$-Magidor $M \prec \frak A$, there is a bound $b_{\frak A}(M) <\kappa$ on the order-type of $N \cap \theta$ whenever $N \prec \frak A$ is a $\kappa$-Magidor end-extension of $M$.  

Let $\lambda<\theta$ be such that $V_\lambda \prec \frak A$ and $\cf(\lambda) \geq \kappa$.
%(or $\cf(\lambda) = \lambda$ if $\theta$ is Mahlo).  
Let $\delta_0 \geq \lambda$ witness ambitiousness.  Let $\delta \geq \delta_0$ be such that  $V_\delta \prec \frak A$ and $\cf(\delta) \geq \kappa$.
%  (or $\cf(\delta) = \delta$ if $\theta$ is Mahlo).    
Let $\calW$ be a normal ultrafilter on $\p_\kappa(V_\delta)$ such that for all $f : \p_\kappa(V_\lambda) \to \kappa$, $\{ z \in \p_\kappa(V_\delta) : f(z\cap V_\lambda) < \ot(z) \} \in \calW$.  Let $\calU$ be the projection of $\calW$ to $\p_\kappa(V_\lambda)$, and let $j : V \to M$ be the ultrapower embedding via $\calU$.  Then in $M$, $j[V_\lambda]$ is a $j(\kappa)$-Magidor elementary submodel of $j(V_\lambda) \prec j(\frak A)$.  
It is ${<}\kappa$-closed by Lemma \ref{magclosure}.  Let $b = b_{\frak A} \restriction \p_\kappa(V_\lambda)$, and let $\xi = j(b)(j[V_\lambda])$.
Let $i : V \to N$ be the ultrapower map via $\calW$, and let $k : M \to N$ be the factor map.  Then $\delta > k(\xi)$.
In $N$, $k(\xi) = i(b_{\frak A})(i[V_\lambda])$ is a bound on the order-types of $i(\kappa)$-Magidor end-extensions of $i[V_\lambda]$ that are elementary in $i(\frak A)$.  But $i[V_\delta]$ is an $i(\kappa)$-Magidor elementary submodel of $i(\frak A)$, $i[V_\delta] \sqsupseteq i[V_\lambda]$, and $\ot(i[V_\delta] \cap \theta) = \delta > k(\xi)$.  This contradiction shows that (\ref{mag-ee}) holds.

Now suppose (\ref{mag-ee}).  Suppose $\kappa \leq \lambda <\theta$, and let $\frak A$ be any structure in a countable language expanding $(V_\theta,\in,\kappa,\lambda)$.  Let $M \prec \frak A$ be a $\kappa$-Magidor model that can be end-extended to other $\kappa$-Magidor models elementary in $\frak A$ of arbitrarily high order-type below $\kappa$.  There is $\gamma < \kappa$, a structure $\bar{\frak A}$ on $V_\gamma$, and an elementary $j : \bar{\frak A} \to \frak A$, with $j[V_\gamma] = M$.  Let $\la\bar\kappa,\bar\lambda\ra = j^{-1}(\la\kappa,\lambda\ra)$.  Then we can define a normal ultrafilter $\calU$ on $\p_{\bar\kappa}(\bar\lambda)$ by $X \in \calU$ iff $j[\bar\lambda] \in j(X)$.  Of course, $\calU \in V_\gamma$.  Let $j_\calU : V_\gamma \to P = \Ult(V_\gamma,\calU)$ be the ultrapower embedding.  As usual, there is a factor map $\ell_0 : P \to V_\theta$ defined by $\ell_0([f]_{\calU}) = j(f)(j[\bar\lambda])$, with $j = \ell_0 \circ j_\calU$.  Let $\xi = \sup \ell_0[j_\calU(\bar\kappa)]$, which is above $\bar\lambda$ since $j_\calU(\bar\kappa) > \bar\lambda$.  Since $\gamma<\kappa$ and $\kappa$ is regular, $\xi <\kappa$.

Now let $N \prec \frak A$ be a $\kappa$-Magidor end-extension of $M$ with $\ot(N \cap \theta) > \xi$.  Let $\eta<\kappa$ be such that $N \cong V_\eta$.  Let $i : V_\eta \to V_\theta$ be the inverse of the transitive collapse map.  Then $i$ is an extension of $j$.  Let $\xi\leq\bar\delta<\eta$.  Let $\calW$ be the normal ultrafilter on $\p_{\bar\kappa}(\bar\delta)$ defined by $X \in \calW$ iff $i[\bar\delta] \in i(X)$.  Then $\calU$ is the projection of $\calW$ via the map $\pi : z \mapsto z \cap \bar\lambda$.  Let $j_{\calW} : V_\eta \to Q = \Ult(V_\eta,\calW)$ be the ultrapower embedding, let $k : P \to Q$ be the factor map $[f]_\calU \mapsto [f \circ \pi]_\calW$, and let $\ell_1 : Q \to V_\theta$ be the map $[f]_\calW \mapsto i(f)(i[\bar\delta])$.  We have the following commutative diagram:
\begin{diagram}
V_\gamma	&\rTo^{j_\calU}	&P		&\rTo^{\ell_0}	&V_\theta	\\
\dTo^\id		&			&\dTo^k	&\ruTo^{\ell_1}	&		\\
V_\eta		&\rTo^{j_\calW}	&Q		&			&			
\end{diagram}
We have that:
\begin{align*}
\bar\delta	&\geq \sup\{ \ell_0(\alpha) : \alpha < j_\calU(\bar\kappa) \} \\
		&=  \sup\{ \ell_1(k(\alpha)) : \alpha < j_\calU(\bar\kappa) \} \\
		&\geq \sup\{ k(\alpha) : \alpha < j_\calU(\bar\kappa) \}
\end{align*}
The relation $\bar\delta \geq \sup k[ j_\calU(\bar\kappa) ]$ is equivalent to the statement that for all $f : \p_{\bar\kappa}(\bar\lambda) \to \bar\kappa$, $\{ z \in \p_{\bar\kappa}(\bar\delta) : f(z \cap \bar\lambda) < \ot(z) \} \in \calW$, which can be computed in $V_\eta$.  Note that to arrange this property, we only needed to take $\bar\delta \geq \xi$.  By the elementarity of $i = \ell_1 \circ j_\calW$, %$i(\calU)$ is a normal ultrafilter on $\p_\kappa(\lambda)$ such that 
for every $\delta \geq i(\xi)$ in $V_\theta$, there is a normal ultrafilter $\calW'$ on $\p_\kappa(\delta)$ %projecting to $i(\calU)$ 
such that for all $f : \p_{\kappa}(\lambda) \to \kappa$, $\{ z \in \p_{\kappa}(\delta) : f(z \cap \lambda) < \ot(z) \} \in \calW'$.
% if $k : \Ult(V_\theta,i(\calU)) \to \Ult(V_\theta,\calW)$ is the factor map, then $\delta \geq \sup k[ j_{i(\calU)}(\kappa) ]$.
\end{proof}

In contrast to the class of countable models, Magidor models are more constrained in their end-extendibility.  

\begin{prop}
\label{maxmag}
Suppose $\kappa<\theta$ are inaccessible and $V_\theta \models \kappa$ is supercompact.  Then there is a set $B \subseteq \kappa$ and a function $f : B \to \kappa$ such that:
\begin{enumerate}
\item $M \cap \kappa \in B$ for stationary many $\kappa$-Magidor $M \prec V_\theta$.
\item For all $\kappa$-Magidor $M \prec V_\theta$ such that $M \cap \kappa \in B$, $\ot(M \cap \theta) \leq f(M \cap \kappa)$.
% and there is no $\kappa$-Magidor $N \prec V_\theta$ such that $M \sqsubseteq N$ and $M \not= N$.
\end{enumerate}
\end{prop}

\begin{proof}
Let $B = \{ \alpha < \kappa : V_\kappa \models$ ``$\alpha$ is not supercompact''$\}$.  For $\alpha \in B$, let $f(\alpha)$ be the least $\beta$ such that $\alpha$ is not $\beta$-supercompact.

Suppose $\kappa$ is $\lambda$-supercompact
%and $2^\eta \leq \lambda$ for all cardinals $\eta<\lambda$.  
Let $\calU,\calW$ be two normal ultrafilters on $\p_\kappa(\lambda)$.  If $\calU \in \Ult(V,\calW)$, then $j_\calU(\kappa) < j_\calW(\kappa)$ since $j_\calW(\kappa)$ is inaccessible in $\Ult(V,\calW)$, and $\Ult(V,\calW)$ can compute $\Ult(\kappa,\calU)$.  Thus if $\calU$ has $j_\calU(\kappa)$ as small as possible, then $\kappa$ is not $\lambda$-supercompact in $\Ult(V,\calU)$.  
%By the $\lambda$-closure of $\Ult(V,\calU)$, $\kappa$ is $\eta$-supercompact in $\Ult(V,\calU)$ for all $\eta<\lambda$.  
Thus $\kappa \in j_\calU(B)$ and $j_\calU(f)(\kappa) \leq \lambda$.

Let $\frak A$ be a structure on $V_\theta$ in a countable language, and let $\lambda$ be such that $\kappa<\lambda< \theta$ and $V_\lambda \prec \frak A$.  Let $\calU$ be a normal ultrafilter on $\p_\kappa(\lambda)$ with $j_\calU(\kappa)$ as small as possible.  In $M = \Ult(V,\calU)$, $N = j_\calU[V_\lambda]$ is a $j_\calU(\kappa)$-Magidor elementary submodel of $j_\calU(\frak A)$.  By the above observations, $\kappa =N\cap j_\calU(\kappa) \in j_\calU(B)$.  By elementarity, there is $\kappa$-Magidor $N' \prec \frak A$ such that $N' \cap \kappa \in B$.

Now suppose $M \prec V_\theta$ is $\kappa$-Magidor and $M \cap \kappa \in B$.  Suppose towards a contradiction that there is a $\kappa$-Magidor $N \prec V_\theta$ such that $M \sqsubseteq N$ and $\ot(N) > f(M \cap \kappa)$.  Then there is an elementary $i : V_\eta \to V_\theta$, with $i[V_\eta] = N$ and $\eta > f(M\cap\kappa)$.  We can define a normal ultrafilter $\calW$ on $\p_{M\cap\kappa}(f(M\cap\kappa))$ by $X \in \calW$ iff $i[f(M\cap\kappa)] \in i(X)$.  This contradicts the fact that $M \cap \kappa$ is not $f(M\cap\kappa)$-supercompact.
\end{proof}

%  Working in $M$, if there were a $j_\calU(\kappa)$-Magidor $N_1 \prec V^M_\theta$ end-extending $N_0$, then there would be an elementary $i : V_\eta^M \to V_\theta^M$, with $i[V_\eta^M] = N_1$ and $\eta > \lambda$.  Then we could define a normal ultrafilter $\calW$ on $\p_\kappa(\lambda)$ by $X \in \calW$ iff $i[\lambda] \in i(X)$.  This contradicts the observation of the previous paragraph that $\kappa$ is not $\lambda = j_\calU(f)(\kappa)$-supercompact in $M$.  By the elementarity of $j_\calU$, there is a  $\kappa$-Magidor $N \prec \frak A$ such that $\ot(N \cap \theta) = f(N \cap \kappa)$ and there is no strictly larger $\kappa$-Magidor $N' \sqsupseteq N$ elementary in $V_\theta$.

%Note that the argument for the above proposition depended on finding a $\kappa$-Magidor model $M$ such that $M \cap \kappa$ and $\kappa$ differ on some large cardinal property; namely, $M \cap \kappa$ lacks a degree of supercompactness.  In \S\ref{limitations}, we will find other constraints on the end-extendibility of Magidor models that do not rely on such tricks.

To situate ambitiousness among the large cardinal notions, we show that it is strictly below almost-hugeness and strictly above the strongest notion found in the literature below almost-huge.  Perlmutter's paper \cite{perlmutter} seems to be the latest word on this.  A cardinal $\kappa$ is called \emph{high-jump} when for some $\lambda \geq \kappa$, there is a normal ultrafilter $\calU$ on $\p_\kappa(\lambda)$ such that $\lambda\geq\sup\{ j_\calU(f)(\kappa) : f : \kappa \to \kappa \}$.  Such an ultrafilter is called a \emph{high-jump measure}.    This notion was first introduced in \cite{srk} and given the name ``high-jump'' by \cite{ha}.  The strongest notion considered in \cite{perlmutter} below almost-huge is \emph{high-jump with unbounded excess closure}.  Perlmutter \cite[Proposition 4.6]{perlmutter} shows that this is equiconsistent with the existence of a cardinal $\kappa$ such that for all sufficiently large $\lambda \geq \kappa$, there is a high-jump measure on $\p_\kappa(\lambda)$.

\begin{prop}
Suppose $\kappa$ is almost-huge with target $\lambda$.  Then $V_\lambda \models$ ``$\kappa$ is ambitious,'' and $V_\kappa \models$ ``There are unboundedly many ambitious cardinals.''  
\end{prop}

\begin{proof}
Let $j : V \to M$ be such that $\crit(j) = \kappa$, $j(\kappa) = \lambda$, and $M^{<\lambda} \subseteq M$.  Suppose $\kappa \leq \alpha < \lambda$.  
Since $\lambda$ is inaccessible, there is $\delta_0 < \lambda$ such that for all $f : \p_\kappa(\alpha) \to \kappa$, $j(f)(j[\alpha])<\delta_0$.
Let $\delta_0\leq\delta<\lambda$ and let $\calU$ be the normal ultrafilter on $\p_\kappa(\delta)$ derived from $j$.  Then for all $f : \p_\kappa(\alpha) \to \kappa$,
$\{ z \in \p_\kappa(\delta) : f(z \cap \alpha) < \ot(z) \} \in \calU.$
The second claim follows from an easy reflection argument.
\end{proof}

\begin{prop}
If $\kappa$ is ambitious, then for all sufficiently large $\lambda \geq \kappa$, there is a high-jump measure on $\p_\kappa(\lambda)$.  Furthermore, the former property has strictly greater consistency strength than the latter.
\end{prop}

\begin{proof}
If $\kappa$ is ambitious, then there is $\delta_0 > \kappa$ such that for all $\delta\geq\delta_0$, there is a normal ultrafilter $\calU$ on $\p_\kappa(\delta)$ such that for all $f : \kappa \to \kappa$, $\{ z \in \p_\kappa(\delta) : f(z \cap \kappa) < \ot(z) \} \in \calU$.  This means that if $j : V \to M$ is the ultrapower embedding via $\calU$, then for all $f : \kappa \to \kappa$, $j(f)(\kappa) < \delta$.  Thus $\calU$ is a high-jump measure.

To prove the strict consistency comparison, we show that there are unboundedly many inaccessible cardinals above an ambitious cardinal $\kappa$.  Then if $\theta$ is any inaccessible above the witness $\delta_0$ for $\kappa$, then $V_\theta \models$ ZFC +  ``For all $\delta\geq\delta_0$, there is a high-jump measure on $\p_\kappa(\delta)$.''  To this end, let $\lambda \geq \kappa$ be arbitrary.  Let $j : V \to M$, $k : M \to N$ be as in the definition of ambitiousness.  There is $\theta$ such that $\lambda < \theta < j(\kappa)$ and $M \models \theta$ is inaccessible.  There is $\delta$ such that $N$ is $\delta$-closed and $k(\theta) < \delta$.  Thus $k(\theta)$ is inaccessible in $V$, and $k(\theta) > \lambda$.
\end{proof}

\section{Forcing with Magidor models}
\label{sec:hopeless}
In this final section, we end on a negative note.  We first describe a variety of circumstances under which forcing with Magidor models as the small type in a Neeman forcing will force the weak Chang's Conjecture, and we demonstrate a certain reversal, yielding a forcing characterization of ambitious cardinals. 
%We also separate $\wcc$ from weak presaturation.
Then we address a claim of Neeman \cite[\S5.1]{neeman} about iterating his forcing and show that ambitious cardinals provide a counterexample.  Finally, we present an argument of Mohammapour that his forcing with Veli\v{c}kovi\'c \cite{mv} does not permit $\wcc(\omega_2,\cof(\omega_1))$.  This seems to close the door on combining compactness and hugeness at $\omega_2$ with existing side-conditions technology.

\subsection{A forcing characterization of ambitiousness}

If $M \prec V_\theta$ is a Magidor model and $\mathbb P \in M$ is a partial order, then $p \in \mathbb P$ is a strong master condition for $M$ if and only if it is an ordinary master condition for $M$.  For suppose $p$ is an ordinary master condition.  If we take a dense $D \subseteq \mathbb P \cap M$ and let $\pi : M \to V_\eta$ be the transitive collapse, then $\pi[D] =E \in V_\eta$ is a dense subset of $\pi(\mathbb P)$.  Then $\pi^{-1}(E) \cap M = D$, and since $p$ is a master condition for $M$, $p$ forces the generic filter $\dot G$ to meet $\pi^{-1}(E)$ in $M$, i.e.\ it forces $\dot G \cap D \not= \emptyset$.

Let us fix the following assumptions and notation for the remainder of this section:

\begin{enumerate}
\item $\mu<\kappa<\theta$ are regular cardinals, and $\theta$ is inaccessible.
\item $\mathbb P \in V_\theta$ is a partial order.
\item $\frak A = \la V_\theta,\in,\mu,\kappa,\mathbb P \ra$.
\item $\calS_0$ is the set of all $\kappa$-Magidor $M \prec \frak A$ that are ${<}(M\cap\kappa)$-closed.
\item $\calT_0$ is the set of all $V_\alpha \prec \frak A$ such that $\cf(\alpha) \geq \kappa$.
\item $\mathbb P$ is proper for $\calS_0$.
\item If $G \subseteq \mathbb P$ is generic, define in $V[G]$:
\begin{enumerate}
\item $\calS = \{ M[G] : M \in \calS_0 \text{ and } G \text{ has a master condition for } M \}$
\item $\calT = \{ W[G] : W \in \calT_0 \}$
\item $\mathbb Q_0 = \mathbb P_{\mu,\calS,\calT}$
\item $\mathbb Q_1 = \mathbb P_{\mu,\calS,\calT}^\mathrm{dec}$  (as defined in \cite{neeman})
\item $\mathbb Q_2 = \mathbb P_{\mu,\calS,\calT}^\mathrm{dec*}$
\end{enumerate}
\end{enumerate}

\begin{claim}
\label{SmeetT}
Suppose $M,N \in \calS_0$, $W \in \calT_0$, and $G \subseteq \mathbb P$ is generic.
\begin{enumerate} 
\item\label{trivmaster} $1_{\mathbb P}$ is a strong master condition for $W$. 
\item\label{grndmeet} $M \cap W \in W \cap \calS_0$.
\item\label{eeprop} If $M \sqsubseteq N$, then $M[G] \in \calS$ iff $N[G] \in \calS$.
\item\label{extmeet1} If $M[G] \in \calS$, and $W \in M$, then $M[G] \cap W[G] = (M \cap W)[G] \in \calS \cap W[G]$.
%\item\label{extmeet2} If $q \in \mathbb Q_0$ and $M[G],N[G] \in q$, then $M[G] \cap N[G] = (M \cap N)[G]$.
%$G$ has master conditions for $M$ and $N$ and $M \cap N \prec \frak A$, then $M[G] \cap N[G] = (M \cap N)[G]$.
% $M[G] \in \calS$, and $W \in M$, then $M[G] \cap W[G] = (M \cap W)[G] \in \calS \cap W[G]$.
%\item\label{elemmeet} If $G$ has master conditions for $M$ and $N$, and $M[G] \cap N[G] \prec V_\theta[G]$, then $M[G] \cap N[G] = (M \cap N)[G]$.
\end{enumerate}
\end{claim}

\begin{proof}
(\ref{trivmaster}) holds since $\p(\mathbb P) \subseteq M$ for all $M \in \calT_0$.

For (\ref{grndmeet}), note that $M \cap W$ is $\kappa$-Magidor since it is an initial segment of $M$, elementary in $\frak A$ by Lemma \ref{rankmeet}, and a member of $W$ by the ${<}\kappa$-closure of $W$.

For (\ref{eeprop}), suppose $M \sqsubseteq N$ are both in $\calS_0$.  Recall that for $X \prec \la H_\theta,\in,\mathbb P \ra$, $p$ is a strong master condition for $X$ iff $p$ forces that $\dot G \cap X$ is generic over $V$ for the partial order $\mathbb P \cap X$.
Let $\la p_\alpha : \alpha < | \mathbb P | \ra$ be an enumeration of $\mathbb P$ in $M$.  Since $M \sqsubseteq N$, $M \cap \mathbb P := \mathbb P^* = \{ p_\alpha : \alpha \in M \cap | \mathbb P | \} = N \cap \mathbb P $.  Thus $p$ is a strong master condition for $M$ or $N$ if and only if $p$ forces that $\dot G \cap \mathbb P^*$ is generic over $V$ for $\mathbb P^*$.

%For (\ref{extmeet}), it is well-known that for $X \prec \la H_\theta,\in,\mathbb P \ra$, $p$ is a master condition for $X$ iff $p \Vdash X[\dot G] \cap \ord = X \cap \ord$.  So if $G$ has master conditions for both $M$ and $N$, then there is $p \in G$ that is a master condition for both.  We have that such $p$ forces 
%$M[G] \cap N[G] \cap \ord = M \cap N \cap \ord$, so $p$ is a master condition for $M \cap N$.
For (\ref{extmeet1}), first note that since $M \cap W \sqsubseteq M$, it follows by (\ref{eeprop}) that $(M \cap W)[G] \in \calS$.  Since $M \cap W \in W$, $(M \cap W)[G] \in W[G]$.  
To show $(M \cap W)[G] = M[G] \cap W[G]$, let $x \in M[G] \cap W[G]$.  Since $M[G] \models W[G] = \{ \tau^G : \tau \in W$ is a $\mathbb P$-name$\}$, there is $\mathbb P$-name $\tau \in M \cap W$ such that $x = \tau^G$.  Thus $M[G] \cap W[G] \subseteq (M \cap W)[G]$.   The reserve inclusion is trivial.
\end{proof}
%For (\ref{extmeet2}), it follows from Claim 2.12 of \cite{neeman} that 

%For (\ref{elemmeet}), by Laver's result \cite{laver}, the ground model is a definable class of $V_\theta[G]$, from parameters common to all elements of $\calS \cup \calT$.  If $x \in M[G] \cap N[G]$, then $M[G] \cap N[G] \models (\exists \tau \in V) x= \tau^G$.  Since $G$ has master conditions for $M$ and $N$, $M[G] \cap V = M$ and $N[G] \cap V= N$, so $M[G] \cap N[G] \cap V = M \cap N$.  Thus for each $x \in M[G] \cap N[G]$, there is a $\mathbb P$-name $\tau \in M \cap N$ such that $x = \tau^G$, showing $M[G] \cap N[G] \subseteq (M \cap N)[G]$.  The reverse inclusion is trivial.

\begin{claim}
\label{ifstat}
Suppose $\calS_0$ is stationary in $V$ and $G \subseteq \mathbb P$ is generic.
\begin{enumerate}
\item\label{statpresS} $\calS$ is stationary in $V[G]$.
\item\label{cofpresT} $\calT = \{ V_\alpha^{V[G]} \prec V_\theta^{V[G]} : \cf(\alpha) \geq \kappa \wedge \mu,\kappa,\mathbb P \in V_\alpha \}$.
%\item\label{appropriate} If $\mu$ is regular in $V[G]$, then $\la \mu,\kappa,\theta,\la V_\alpha : \alpha < \theta \ra,\calS,\calT \ra$ is nice.
\end{enumerate}
\end{claim}

\begin{proof}
For (\ref{statpresS}), consider a name for a function $\dot F : (H_{\theta}^{V[G]})^{<\omega} \to H_\theta^{V[G]}$.  Let $p \in \mathbb P$ be arbitrary.  Let $\theta^*>\theta$ be regular and let $\frak B = \la H_{\theta^*},\in,\mathbb P,\dot F,p \ra$.
Let $M^* \prec \frak B$ be such that $M^* \cap H_\theta = M \in \calS_0$.  By $\calS_0$-properness, let $p' \leq p$ be a master condition for $M$.  Then $p'$ forces that $M[G]$ is closed under $\dot F$ and $M[G] \in \calS$.

For (\ref{cofpresT}), note that by the $\calS_0$-properness of $\mathbb P$, the class of ordinals $\cof({\geq}\kappa)$ is the same in $V[G]$.  It is a standard fact that if $\mathbb P \in V_\alpha$, then $V_\alpha^{V[G]} = V_\alpha[G]$.  If $\mathbb P \in V_\alpha$ and $V_\alpha \prec V_\theta$, then if $V_\alpha[G] \models \varphi(\tau^G)$, this is forced by some $p \in G$.  Thus $V_\alpha \models$ ``$p \Vdash \varphi(\tau)$'', and therefore $V_\theta \models$ ``$p \Vdash \varphi(\tau)$'', so $V_\theta[G] \models \varphi(\tau^G)$.  On the other hand, suppose $V_\alpha[G] \prec V_\theta[G]$.  By Laver's result \cite{laver} that the ground model is a definable class of a forcing extension, there is a translation of formulas $\varphi(\vec x) \mapsto \varphi(\vec x)^{\mathrm{Gr}}$ such that for $\vec a \in V_\alpha^{<\omega}$, $V_\alpha \models \varphi(\vec a)$ if and only if $V_\alpha[G] \models \varphi(\vec a)^{\mathrm{Gr}}$.  Thus for all $\vec a \in V_\alpha^{<\omega}$, $V_\alpha \models \varphi(\vec a)$ iff $V_\alpha[G] \models \varphi(\vec a)^{\mathrm{Gr}}$ iff $V_\theta[G] \models \varphi(\vec a)^{\mathrm{Gr}}$ iff $V_\theta \models \varphi(\vec a)$.
%(\ref{appropriate}) follows from Claim \ref{SmeetT}(\ref{nicemeet}) and (\ref{cofpresT}) above.
\end{proof}

%In $V^{\mathbb P}$, let us denote three types of Neeman forcing as follows:
%\begin{align*}
%\mathbb M_0(\mu,\kappa,\theta) &= \mathbb P_{\mu,\calS,\calT} \\
%\mathbb M_1(\mu,\kappa,\theta) &= \mathbb P_{\mu,\calS,\calT}^\mathrm{dec} \text{\hspace{3mm}(from \cite{neeman})} \\
%\mathbb M_2(\mu,\kappa,\theta) &= \mathbb P_{\mu,\calS,\calT}^\mathrm{dec*} \\
%\end{align*}

\begin{lemma}
\label{forcewcc}
Suppose $V_\theta \models$ ``$\kappa$ is ambitious,'' and $\mathbb P$ preserves the regularity of $\mu$.  Let $A = \{ M \cap \kappa : M \in \calS_0 \}$.  Then for each $n<3$, $\mathbb P * \dot{\mathbb Q}_n$ forces $\wcc(\kappa,A)$.
\end{lemma} 

\begin{proof}
By Claims \ref{SmeetT} and \ref{ifstat} and by the preservation of $\mu$, $\la \calS,\calT \ra$ is appropriate for $\mu$ and $V_\theta[G]$, whenever $G \subseteq \mathbb P$ is generic.
By \cite{neeman}, or Lemma \ref{cards}, $\mathbb P * \dot{\mathbb Q}_n$ preserves the regularity of $\kappa$ and $\theta$ and forces that $\theta = \kappa^+$.
Let $\dot F$ be a $\mathbb P * \dot{\mathbb Q}_n$-name for a function from $\theta^{<\omega}$ to $\theta$, let $\dot f$ be a name for a function from $\kappa$ to $\kappa$, and let $\la p_0,\dot q_0 \ra \in \mathbb P * \dot{\mathbb Q}_n$ be arbitrary.  Let $\frak B = \la H_{\theta^+},\in,\frak A,\dot F,\dot f,\la p_0,\dot q_0\ra,\lhd \ra$, where $\lhd$ is a well-order of $H_{\theta^+}$.  Let $\frak C$ be the result of restricting the Skolem functions for $\frak B$ to $V_\theta$.  Let $F' : V_\theta^{<\omega} \to V_\theta$ be another function.  By Theorem \ref{amb-ee}, there is a $\kappa$-Magidor $M \prec \frak C$ that is ${<}(M\cap\kappa)$-closed, $M$ is closed under $F'$, and for every $\alpha<\kappa$, there is a $\kappa$-Magidor $N \prec \frak C$ such that 
%$N$ is closed under $F'$, 
$M \sqsubseteq N$, $N$ is ${<}(N\cap\kappa)$-closed, and $\ot(N \cap \theta)>\alpha$.

It follows that there are stationary-many $\kappa$-Magidor $M \prec \frak C$ that are ${<}(M\cap\kappa)$-closed and can be end-extended to $\kappa$-Magidor $N \prec \frak C$ with the same closure and such that $\ot(N \cap \theta)$ is arbitrarily large below $\kappa$.  Let $\calE$ be the set of such $M$. For each $\alpha<\theta$ choose $M_\alpha \in \calE$ such that $\alpha \in M_\alpha$.  If $\cf(\alpha) \geq \kappa$, then $\sup(M_\alpha \cap \alpha) < \alpha$.  By Fodor's Lemma and the inaccessibility of $\theta$, there is a stationary $X \subseteq \theta \cap \cof({\geq}\kappa)$ and an $M^*$ such that for every $\alpha \in X$, $M_\alpha \cap \alpha = M^*$.  We may assume that for every $\alpha \in X$, $V_\alpha \prec \frak C$.  It follows that $M^* \prec \frak C$, $M^*$ is $\kappa$-Magidor, and $M^*$ is ${<}(M^*\cap\kappa)$-closed.

By $\calS_0$-properness, let $p_1 \leq p_0$ be a master condition for $M^*$.
$p_1$ forces that $M^*[\dot G] \in \calS$ and thus by \cite{neeman}, or Lemma \ref{spdec}, that $\dot q_0$ can be extended to include the model $M^*[\dot G]$.  Let $\dot q_1$ be a name for such an extension.
Let $\la p_2,\dot q_2 \ra \leq \la p_1,\dot q_1 \ra$ decide the value of $\dot f(M^*\cap\kappa)$, say as $\xi<\kappa$.  
Let $\alpha \in X$ be such that $\la p_2,\dot q_2 \ra \in V_\alpha$.  Let $N \sqsupseteq M_\alpha$ be such that $N$ is $\kappa$-Magidor, $N \prec \frak C$, $N$ is ${<}(N\cap\kappa)$-closed, and $\ot(N \cap \theta)> \xi$.  Note that by Claim \ref{SmeetT}, $p_1$ is a master condition for $N$.

Let $\dot s_2$ be a name for the set of models appearing in $\dot q_2$.  Let $\dot s_3$ be a name such that $\Vdash_{\mathbb P} \dot s_3 = \dot s_2 \cup \{ V_\alpha[\dot G]
,N[\dot G] \}$.  Then $p_2$ forces that $\dot s_3$ is closed under intersections 
since $N \cap V_\alpha = M^*$, and thus by Claim \ref{SmeetT} that $N[\dot G] \cap V_\alpha[\dot G] = M^*[\dot G] \in \dot s_2$, and that if $R \in \dot s_2$, then $V_\alpha[\dot G] \cap R = R$ and $N[\dot G] \cap R = N[\dot G] \cap V_\alpha[\dot G] \cap R = M^*[\dot G] \cap R \in \dot s_2$.
%if some $p_3 \leq p_2$ forces $R[\dot G] \in \dot s_2$, then $p_3$ forces $V_\alpha[\dot G] \cap R[\dot G] = R[\dot G]$, and $N[\dot G] \cap R[\dot G] = N[\dot G] \cap V_\alpha[\dot G] \cap R[\dot G] = M^*[\dot G] \cap R[\dot G]$, 
%and $p_3 \Vdash (M^* \cap R)[G] = M^*[\dot G] \cap R[\dot G] \in \dot s_2$, using Claim \ref{SmeetT}.
Furthermore, it is forced that $\dot s_3 \cap N[\dot G] = (\dot s_3 \cap M^*[\dot G]) \cup \{ V_\alpha[\dot G] \}$, which is forced to be a member of $N[\dot G]$.

If $n = 0$, let $\dot q_3 = \dot s_3$.
If $n=1$ and $\dot h$ is a name for the decoration of $\dot q_2$, let $\dot h'$ be a name for the extension of $\dot h$ that assigns the empty set to $V_\alpha[\dot G]$ and $N[\dot G]$, and let $\dot q_3$ be a name for $\la \dot s_3,\dot h' \ra$.
If $n = 2$ and $\dot h$ is a name for the decoration of $\dot q_2$, let $\dot q_3$ be a name for $\la \dot s_3,\dot h \ra$.
In each case, $\la p_2,\dot q_3 \ra \leq \la p_2,\dot q_2 \ra$, and $\la p_2,\dot q_3 \ra$ is a master condition for $\sk^{\frak B}(N)$.
% $\dot p_3$ be a name for $\la \dot s_3,\dot h \ra$.  Then $\la q_1,\dot p_3 \ra \leq \la q_1,\dot p_2 \ra$, and $\la q_1,\dot p_3 \ra$ is a master condition for $\sk^{\frak B}(N)$.  
Thus $\la p_2,\dot q_3 \ra$ forces that $N$ is closed under $\dot F$ and that $\dot f(N \cap \kappa) < \ot(N \cap \theta)$.  By the arbitrariness of $\dot F$, $\dot f$, and $\la p_0,\dot q_0\ra$, $\wcc(\kappa,A)$ is forced.
\end{proof}

\begin{theorem}
Suppose $\mathbb P * \dot{\mathbb Q}_2$ preserves the class $\cof({\geq}\mu)$.  Then $\mathbb P * \dot{\mathbb Q}_2$ forces $\wcc(\mu^+,\cof(\mu))$ if and only if $V_\theta \models$ ``$\kappa$ is ambitious.''
\end{theorem}

\begin{proof}
Suppose $V_\theta \models$ ``$\kappa$ is ambitious.''  By Lemma \ref{forcewcc}, $\wcc(\kappa,A)$ holds, where $A = \{ M \cap \kappa : M \in \calS_0 \}$.   For every $\kappa$-Magidor $M \prec V_\theta$, $V \models \mu < \cf(M \cap \kappa) < \kappa$, so since all cardinals between $\mu$ and $\kappa$ are collapsed, $\mathbb P * \dot{\mathbb Q}_2$ forces that $\cf(M \cap \kappa) = \mu$.

For the other direction, suppose  $V_\theta \models$ ``$\kappa$ is not ambitious.''

\textbf{Case 1}:  $V_\theta \models$ ``$\kappa$ is not supercompact.''  Then there is $\lambda_0 \in [\kappa,\theta)$ such that for all $\lambda \geq \lambda_0$, there is no normal fine ultrafilter on $\p_\kappa(\lambda)$.  It follows that for every $\lambda \in [\lambda_0,\theta)$, there is no $\kappa$-Magidor $M \prec V_\theta$ with $\lambda \in M$.  This is because if there were such an $M$, then there would be an elemetnary $i : V_\eta \to V_\theta$ for some $\eta<\kappa$, with $i(\crit(i)) = \kappa$.  If $i(\bar\kappa) = \kappa$ and $i(\bar\lambda) = \lambda$, then we could define a normal fine ultrafilter $\calU$ on $\p_{\bar\kappa}(\bar\lambda)$ by $X \in \calU$ iff $i[\bar\lambda] \in i(X)$.  But then $\calU \in V_\eta$, and by elementarity, $V_\theta \models$ ``$\kappa$ is $\lambda$-supercompact,'' a contradiction.

Let $G \subseteq \mathbb P$ be generic.  In $V[G]$, for any $q \in\mathbb Q_2$
the only models included in $p$ of rank $> \lambda_0$ are transitive models.  Thus if $V_\alpha \in V_\beta$ are two consecutive transitive models, with $\alpha > \lambda_0$, and $H \subseteq \mathbb Q_2$
is generic over $V[G]$, then the union of the decorations attached to $V_\alpha$ that appear in $H$ constitutes a surjection from $\mu$ to $V_\beta$.  But notice also that Corollary \ref{thetacc} still holds for $\mathbb Q_2$,
and so $\theta = \mu^+$ in $V[G][H]$.  Now in $V$, let $f : \theta \to \theta$ be the function $\beta \mapsto \beta^+$.  Let $\alpha<(\theta^+)^V = (\theta^+)^{V[G][H]}$ and let $\sigma : \theta \to \alpha$ be a sujrection in $V$.  For all $\beta < \theta$, $\ot(\sigma[\beta]) <  (\beta^+)^V = f(\beta)$.  Thus in $V[G][H]$, $f$ dominates each canonical function pointwise, and so $\wcc(\theta,A)$ fails for every $A \subseteq \theta$.

\textbf{Case 2}:  $V_\theta \models$ ``$\kappa$ is supercompact but not ambitious.''  In this case, $\calS_0$ is stationary in $V$,
so $\mathbb P * \dot{\mathbb Q}_2$
forces that $\kappa = \mu^+$ and $\theta = \mu^{++}$.  By Theorem \ref{amb-ee}, there is a structure $\frak B$ on $V_\theta$ in a countable language and a function $b : V_\theta \to \kappa$ such that for all $\kappa$-Magidor $M \prec \frak B$, if $N \prec \frak B$ is $\kappa$-Magidor and $M \sqsubseteq N$, then $\ot(N \cap \theta) < b(M)$.  We may assume that $\frak B$ has definable Skolem functions.

Let $G*H \subseteq \mathbb P * \dot{\mathbb Q}_2$
be generic.  $H$ yields a chain of models $\la M_\alpha : \alpha < \theta \ra$.  
Since the set of $M \in \calS$ that appear in $H$ is stationary, there are stationary-many $M$ that appear in $H$ that are closed under the Skolem functions for $\frak B$.
For $\alpha < \kappa$ such that $\alpha = M \cap \kappa$ for some model $M$ appearing in $H$ such that $M \cap V \prec \frak B$, let $M_\alpha$ be the first such model in the chain.  We claim that for all $N$ appearing in $H$ such that $N \cap V \prec \frak B$ and $N \cap \kappa = \alpha$, we have $M_\alpha \sqsubseteq N$.  If $N$ is such a model, then $M_\alpha \cap N$ also appears in $H$, and since $M_\alpha \cap N \cap V \prec \frak B$ and $M_\alpha \cap N \cap \kappa = \alpha$, we must have that $M_\alpha \cap N = M_\alpha$ by the minimality of $M_\alpha$.
If $N \not= M_\alpha$, then since $M_\alpha \notin N$, there must be transitive models appearing between them.   
Let $W$ be the least such transitive model.  Then $W \cap N$ also appears in $H$.  If $W \cap N \not= M_\alpha$ then $W \cap N$ occurs either before or after $M_\alpha$.  If were to occur after, then since there are no transitive models between $M_\alpha$ and $W \cap N$, we would have $M_\alpha \in W \cap N$.  But this is impossible, since $M_\alpha \cap \kappa = W \cap N \cap \kappa$.
But also $W \cap N$ cannot occur before $M_\alpha$, since $M_\alpha \cap N = M_\alpha$, so $M_\alpha \subseteq W \cap N$ and thus $\rank(M_\alpha) \leq \rank(W \cap N)$.  Thus $M_\alpha = W \cap N  \sqsubseteq N$.

In $V[G][H]$, define $f : \kappa \to \kappa$ by $f(\alpha) = b(M_\alpha \cap V)$ if $M_\alpha$ is defined, and otherwise $f(\alpha) = 0$.  Now let $X \subseteq \kappa \cap \cof(\mu)$ be a stationary set in $V[G][H]$.  Since $\mathbb P$ preserves that $\cf(M \cap \kappa) \geq \mu$ for all $M \in \calS$, Lemma \ref{*proj} implies that the set of $M$ appearing in $H$ with $M \cap \kappa \in X$ is stationary.  Let $\kappa<\gamma < \theta$, and let $\sigma : \kappa \to \gamma$ be a bijection in $V[G][H]$.  
Let $M$ be a model appearing in $H$ such that $M \cap V \prec \frak B$, $\gamma \in M$, $M$ is closed under $\sigma$ and $\sigma^{-1}$, and $M \cap \kappa = \alpha \in X$.
% is closed under the Skolem functions for $\frak B$, 
%Let $M \prec H_\theta^{V[G][H]}$ be such that $M$ is closed under the Skolem functions for $\frak B$, $M \cap V[G]$ appears in $H$, $M \cap \kappa = \alpha \in X$, and $\{ f,\sigma \} \in M$.  
There are $N,N_\alpha \in \calS_0$ such that $N[G] = M$, $N_\alpha[G] = M_\alpha$, $M \cap V = N$, and $M_\alpha \cap V = N_\alpha$.
Since $N_\alpha \sqsubseteq N$, $\ot(M \cap \theta) < b(N_\alpha) = f(\alpha)$.  If $\psi_\gamma$ is the canonical function for $\gamma$ defined using $\sigma$, then $\psi_\gamma(\alpha) = \ot(\sigma[\alpha]) = \ot(M\cap\gamma)< \ot(M \cap \theta) < f(\alpha)$.  Since $X$ was an arbitrary stationary subset of $\kappa \cap \cof(\mu)$, the set of $\beta \in \kappa \cap \cof(\mu)$ such that $\psi_\gamma(\beta) \geq f(\beta)$ is nonstationary.  Thus $f$ dominates all canonical functions modulo clubs on $\kappa \cap \cof(\mu)$, so $\wcc(\kappa,\cof(\mu))$ fails.
\end{proof}

We note that the hypothesis of the above theorem is satisfied when $\mathbb P$ is the trivial forcing.

\begin{prop}
Suppose $\mathbb P * \dot{\mathbb Q}_2$ preserves $\cof({\geq}\mu)$.
Then $\mathbb P * \dot{\mathbb Q}_2$ forces that $\ns_\kappa \rest \cof(\mu)$ is \textbf{not} weakly presaturated.
\end{prop}

\begin{proof}
If $V_\theta \models$ ``$\kappa$ is not supercompact,'' then the conclusion follows from the above theorem.  So assume $\kappa$ is supercompact in $V_\theta$.  Let $B \subseteq \kappa$ and $f : B \to \kappa$ be as in Proposition \ref{maxmag}, so that $\calB = \{ M \in \calS_0 : M \cap \kappa \in B \}$ is stationary, and whenever $M \in \calB$, then $\ot(M \cap \theta) \leq f(M \cap \kappa)$.
If $G * H \subseteq \mathbb P * \dot{\mathbb Q}_2$ is generic, then there will be stationary-many $M \in \calB$ such that $M[G]$ appears in $H$.

Using Corollary \ref{wps-wcc}, it suffices to show that $\wcc(\kappa,B)$ fails.  Towards a contradiction, suppose that there is $\gamma<\theta$ and a bijection $\sigma : \kappa \to \gamma$ such that $B' = \{ \alpha \in B : f(\alpha) < \ot(\sigma[\alpha]) \}$ is stationary.  By Lemma \ref{*proj}, there is $M \in \calS_0$ such that $M[G]$ appears in $H$, $M \cap \kappa \in B'$, and $M$ is closed under $\sigma$ and $\sigma^{-1}$.  
But then $\ot(\sigma[M \cap \kappa]) = \ot(M \cap \gamma) < \ot(M \cap \theta) \leq f(M \cap \kappa)$, a contradiction.
\end{proof}

\begin{question*}
Under some large cardinal hypothesis, can $\mathbb P * \dot{\mathbb Q}_2$, or some similar poset, force the existence of a weakly presaturated ideal on $\omega_2$?
\end{question*}

\subsection{A claim of Neeman}

At the end of Section 5.1 of \cite{neeman}, Neeman gave a brief description of an argument that if $\kappa$ is supercompact and $\theta>\kappa$ is weakly compact, then the tree property can be forced simultaneously at $\omega_2$ and $\omega_3$ by a two-step iteration of the form $\mathbb P * \dot{\mathbb Q}$, where $\mathbb P$ is a finite-conditions two-types poset making $\kappa = \omega_2$, and $\dot{\mathbb Q}$ is a countable-conditions two-types poset of size $\theta$.  In private correspondence, he clarified that $\dot{\mathbb Q}$ should indeed use former $\kappa$-Magidor models as the small type.  Since $\mathbb P$ forces $2^\omega = \kappa$, the small models of $\mathbb Q$ will not be countably closed.  Thus, as stated in \cite{neeman}, ``preservation of $\omega_1$ requires a special argument.''  Neeman also clarified that this special argument is that in $V^{\mathbb P}$, $\mathbb Q$ is completely proper for a stationary class of countable models.  Specifically, if $\theta^*>\theta$ is regular, $N \prec H_{\theta^*}$ is countable model in $V$, $\mathbb P * \dot{\mathbb Q} \in N$, $G \subseteq \mathbb P$ is generic, and $N \cap V_\kappa$ appears in $G$, then for any $q \in N[G] \cap \mathbb Q$, there is $q' \leq q$ such that $\{ r \in N[G] \cap \mathbb Q : q' \leq r \}$ is a $\mathbb Q$-generic filter over $N[G]$.  In particular, by a standard argument, $\mathbb Q$ will not add $\omega$-sequences of ordinals.  This complete properness claim also played a key role in the argument for the tree property.

However, Theorem \ref{wccsquare} combined with Lemma \ref{forcewcc} shows that, setting $\mu = \omega_1$ and taking $\kappa$ and $\theta$ sufficiently large, if the second stage $\mathbb Q$ is one of our posets $\mathbb Q_n$ for $n<3$, and it preserves $\omega_1$ and doesn't add $\theta$-many new reals over $V^{\mathbb P}$, then it forces $\square^*_{\omega_1}$, which contradicts the tree property at $\omega_2$.

\subsection{Failure of $\wcc(\omega_2,\cof(\omega_1))$ in the Mohammadpour-Veli\v{c}kovi\'c model}

We present here, for those familiar with the details of \cite{mv}, a sketch of an argument due to Mohammadpour that in a forcing extension by the virtual-models poset of \cite{mv}, there is a function $g : \omega_2 \to \omega_2$ that dominates all canonical functions on $\omega_2$ modulo clubs on $\omega_2\cap\cof(\omega_1)$. 

The forcing $\mathbb P^\kappa_\lambda$ uses a supercompact $\kappa$ and an inaccessible $\lambda>\kappa$.  The poset preserves $\omega_1$, turns $\kappa$ into $\omega_2$, and turns $\lambda$ into $\omega_3$.  The conditions in their poset are finite sequences of countable and $\kappa$-Magidor models, not necessarily elementary in $V_\lambda$, but possibly sets that look like partial transitive collapses of elementary submodels of $V_\lambda$.  See \cite{mv} for details.  

They define $E = \{ \alpha<\lambda : \kappa\in V_\alpha\prec V_\lambda \}$.  For each $\alpha \in E$, there is a projection of $\mathbb P^\kappa_\lambda$ to what we might call the ``$\alpha^{th}$ level,'' $\mathbb P^\kappa_\alpha$, and these projections commute.  For a generic $G \subseteq \mathbb P^\kappa_\lambda$, let $G_\alpha$ denote the projection of $G$ to $\mathbb P^\kappa_\alpha$.  Lemma 4.37 of \cite{mv} says that if $G \subseteq \mathbb P^\kappa_\lambda$ is generic, $\alpha \in E$, $\cf(\alpha)<\kappa$, and $C_\alpha(G) = \{ \sup(M \cap \kappa) : M$ is a model that is ``active at $\alpha$'' and appears in $G_\alpha \}$, then $C_\alpha(G)$ is a club in $\kappa$.  
%Moreover, if $\alpha<\beta<\lambda$ are two such cardinals, then $C_\beta(G) \setminus C_\alpha(G)$ is bounded in $\kappa$.  
The basic definitions of the partial order imply that for each $\alpha \in E$ and $\gamma<\kappa$, there is at most one model $M_\gamma^\alpha$ that is active at $\alpha$, appears in $G_\alpha$, and has $M \cap \kappa = \gamma$.  Furthermore, if $\alpha<\beta$ are in $E$, then for all but boundedly-many $\gamma<\kappa$, $M^\beta_\gamma$ is active at $\alpha$.  If $M^\alpha_\gamma$, $M^\beta_\gamma$ are both defined and active at $\alpha$, then $M^\alpha_\gamma \cong M^\beta_\gamma$, since $M^\alpha_\gamma$ is a partial transitive collapse of $M^\beta_\gamma$.

In $V[G]$, we define representatives of cofinally-many of the canonical functions on $\kappa$ as follows.  For $\alpha \in E$ of cofinality $<\kappa$ and $\gamma \in C_\alpha(G)$ of cofinality $\omega_1$, let $f_\alpha(\gamma) = \ot(M_\gamma^\alpha \cap \alpha)$.  The set of $\kappa$-Magidor models active at $\alpha$ and appearing in $G_\alpha$ is a $\subseteq$-increasing chain of length $\kappa$, continuous at limits of uncountable coflinality, whose union covers $V_\alpha$.  If $\sigma : \kappa \to \alpha$ is any surjection in $V[G]$, then for club-many $\gamma<\kappa$ of cofinality $\omega_1$, $\sigma[\gamma] = M_\gamma^\alpha \cap \alpha$.  Thus $f_\alpha$ represents the $\alpha^{th}$ canonical function on $\cof(\omega_1)$.

Also, define $g_\alpha : C_\alpha(G) \to \kappa$ by $g_\alpha(\gamma) = \ot(M^\alpha_\gamma \cap \lambda)$.  By the above remarks, if $\alpha<\beta$ are in $E$ and have cofinality $<\kappa$, then $g_\alpha(\gamma) = g_\beta(\gamma)$ for all but boundedly-many $\gamma \in  C_\alpha(G) \cap C_\beta(G) \cap \cof(\omega_1)$.  Furthermore, it follows from an easy density argument that for all $\alpha \in E$ of cofinality $<\kappa$ and all but boundedly-many $\gamma<\kappa$, $\alpha\in M^\alpha_\gamma$, and thus  $M^\alpha_\gamma \cap V_\alpha$ is a proper subset of $M^\alpha_\gamma$.  Thus if $g = g_{\alpha_0}$, where $\alpha_0$ is the minimum point of $E$, then for all $\beta<\lambda$, $g$ is greater than the $\beta^{th}$ canonical function at club-many points of cofinality $\omega_1$.
\bibliographystyle{amsplain.bst}
\bibliography{cites.bib}

\providecommand{\bysame}{\leavevmode\hbox to3em{\hrulefill}\thinspace}
\providecommand{\MR}{\relax\ifhmode\unskip\space\fi MR }
\providecommand{\MRhref}[2]{%
  \href{http://www.ams.org/mathscinet-getitem?mr=#1}{#2}
}
\providecommand{\href}[2]{#2}
\begin{thebibliography}{10}

\bibitem{adolf}
Dominik~T. Adolf, \emph{Some basic thoughts on the cofinalities of {C}hang
  structures with an application to forcing}, MLQ Math. Log. Q. \textbf{67}
  (2021), no.~3, 354--358. \MR{4370211}

\bibitem{ha}
Arthur~W. Apter and Joel~David Hamkins, \emph{Universal indestructibility},
  Kobe J. Math. \textbf{16} (1999), no.~2, 119--130. \MR{1745027}

\bibitem{clav}
Benjamin Claverie and Ralf Schindler, \emph{Woodin's axiom {$(\ast)$}, bounded
  forcing axioms, and precipitous ideals on {$\omega_1$}}, J. Symbolic Logic
  \textbf{77} (2012), no.~2, 475--498. \MR{2963017}

\bibitem{coxeskew}
Sean {Cox} and Monroe {Eskew}, \emph{{Compactness versus hugeness at successor
  cardinals}}, arXiv e-prints (2020), arXiv:2009.14245, to appear in J. Math.
  Log.

\bibitem{cox}
Sean~D. Cox, \emph{Chang's conjecture and semiproperness of nonreasonable
  posets}, Monatsh. Math. \textbf{187} (2018), no.~4, 617--633. \MR{3861321}

\bibitem{8fold}
James Cummings, Sy-David Friedman, Menachem Magidor, Assaf Rinot, and Dima
  Sinapova, \emph{The eightfold way}, J. Symb. Log. \textbf{83} (2018), no.~1,
  349--371. \MR{3796288}

\bibitem{dd}
Oliver Deiser and Dieter Donder, \emph{Canonical functions, non-regular
  ultrafilters and {U}lam's problem on {$\omega_1$}}, J. Symbolic Logic
  \textbf{68} (2003), no.~3, 713--739. \MR{2000073}

\bibitem{dl}
Hans-Dieter Donder and Jean-Pierre Levinski, \emph{Some principles related to
  {C}hang's conjecture}, Ann. Pure Appl. Logic \textbf{45} (1989), no.~1,
  39--101. \MR{1024901}

\bibitem{eh}
Monroe Eskew and Yair Hayut, \emph{On the consistency of local and global
  versions of {C}hang's conjecture}, Trans. Amer. Math. Soc. \textbf{370}
  (2018), no.~4, 2879--2905. \MR{3748588}

\bibitem{foreman}
Matthew Foreman, \emph{Ideals and generic elementary embeddings}, Handbook of
  set theory. {V}ols. 1, 2, 3, Springer, Dordrecht, 2010, pp.~885--1147.
  \MR{2768692}

\bibitem{fm}
Matthew Foreman and Menachem Magidor, \emph{Large cardinals and definable
  counterexamples to the continuum hypothesis}, Ann. Pure Appl. Logic
  \textbf{76} (1995), no.~1, 47--97. \MR{1359154}

\bibitem{friedman}
Sy-David Friedman, \emph{Forcing with finite conditions}, Set theory, Trends
  Math., Birkh\"{a}user, Basel, 2006, pp.~285--295. \MR{2267153}

\bibitem{jech}
Thomas Jech, \emph{Set theory}, Springer Monographs in Mathematics,
  Springer-Verlag, Berlin, 2003, The third millennium edition, revised and
  expanded. \MR{1940513}

\bibitem{jensen}
R.~Bj\"{o}rn Jensen, \emph{The fine structure of the constructible hierarchy},
  Ann. Math. Logic \textbf{4} (1972), 229--308; erratum, ibid. 4 (1972), 443,
  With a section by Jack Silver. \MR{309729}

\bibitem{kanamori}
Akihiro Kanamori, \emph{The higher infinite}, second ed., Springer Monographs
  in Mathematics, Springer-Verlag, Berlin, 2003, Large cardinals in set theory
  from their beginnings. \MR{1994835}

\bibitem{kunen}
Kenneth Kunen, \emph{Saturated ideals}, J. Symbolic Logic \textbf{43} (1978),
  no.~1, 65--76. \MR{495118}

\bibitem{ls}
Paul Larson and Saharon Shelah, \emph{Bounding by canonical functions, with
  {CH}}, J. Math. Log. \textbf{3} (2003), no.~2, 193--215. \MR{2030084}

\bibitem{laver}
Richard Laver, \emph{Certain very large cardinals are not created in small
  forcing extensions}, Ann. Pure Appl. Logic \textbf{149} (2007), no.~1-3,
  1--6. \MR{2364192}

\bibitem{magidor}
M.~Magidor, \emph{On the role of supercompact and extendible cardinals in
  logic}, Israel J. Math. \textbf{10} (1971), 147--157. \MR{295904}

\bibitem{mitchell}
William Mitchell, \emph{Aronszajn trees and the independence of the transfer
  property}, Ann. Math. Logic \textbf{5} (1972/73), 21--46. \MR{313057}

\bibitem{mitchellapp}
William~J. Mitchell, \emph{{$I[\omega_2]$} can be the nonstationary ideal on
  {${\rm Cof}(\omega_1)$}}, Trans. Amer. Math. Soc. \textbf{361} (2009), no.~2,
  561--601. \MR{2452816}

\bibitem{mv}
Rahman Mohammadpour and Boban Veli\v{c}kovi\'{c}, \emph{Guessing models and the
  approachability ideal}, J. Math. Log. \textbf{21} (2021), no.~2, Paper No.
  2150003, 35. \MR{4290492}

\bibitem{neeman}
Itay Neeman, \emph{Forcing with sequences of models of two types}, Notre Dame
  J. Form. Log. \textbf{55} (2014), no.~2, 265--298. \MR{3201836}

\bibitem{perlmutter}
Norman~Lewis Perlmutter, \emph{The large cardinals between supercompact and
  almost-huge}, Arch. Math. Logic \textbf{54} (2015), no.~3-4, 257--289.
  \MR{3324126}

\bibitem{sakai}
Hiroshi Sakai, \emph{Semiproper ideals}, Fund. Math. \textbf{186} (2005),
  no.~3, 251--267. \MR{2191239}

\bibitem{shelah78}
Saharon Shelah, \emph{On successors of singular cardinals}, Logic {C}olloquium
  '78 ({M}ons, 1978), Studies in Logic and the Foundations of Mathematics,
  vol.~97, North-Holland, Amsterdam-New York, 1979, pp.~357--380. \MR{567680}

\bibitem{shelah91}
\bysame, \emph{Reflecting stationary sets and successors of singular
  cardinals}, Arch. Math. Logic \textbf{31} (1991), no.~1, 25--53. \MR{1126352}

\bibitem{shelah}
\bysame, \emph{Cardinal arithmetic}, Oxford Logic Guides, vol.~29, The
  Clarendon Press, Oxford University Press, New York, 1994, Oxford Science
  Publications. \MR{1318912}

\bibitem{shioya}
Masahiro Shioya, \emph{Easton collapses and a strongly saturated filter}, Arch.
  Math. Logic \textbf{59} (2020), no.~7-8, 1027--1036. \MR{4159767}

\bibitem{srk}
Robert~M. Solovay, William~N. Reinhardt, and Akihiro Kanamori, \emph{Strong
  axioms of infinity and elementary embeddings}, Ann. Math. Logic \textbf{13}
  (1978), no.~1, 73--116. \MR{482431}

\bibitem{todo89}
Stevo Todor\v{c}evi\'{c}, \emph{Partition problems in topology}, Contemporary
  Mathematics, vol.~84, American Mathematical Society, Providence, RI, 1989.
  \MR{980949}

\bibitem{boban}
Boban Veli\v{c}kovi\'{c}, \emph{P{FA} and precipitousness of the nonstationary
  ideal}, Proc. Amer. Math. Soc. \textbf{146} (2018), no.~2, 791--802.
  \MR{3731712}

\bibitem{woodin}
W.~Hugh Woodin, \emph{The axiom of determinacy, forcing axioms, and the
  nonstationary ideal}, revised ed., De Gruyter Series in Logic and its
  Applications, vol.~1, Walter de Gruyter GmbH \& Co. KG, Berlin, 2010.
  \MR{2723878}

\end{thebibliography}

\end{document}